\documentclass[12pt]{article}
\usepackage{amsmath,amsthm,amssymb}
\usepackage{enumerate}
\usepackage{graphicx}
\usepackage{amsfonts}
\usepackage[breaklinks]{hyperref}
\usepackage{float}
\newcommand{\indic}{1\hspace{-0.15cm}1}

\newtheorem{theorem}{Theorem}[section]
\newtheorem{lemma}[theorem]{Lemma}
\newtheorem{proposition}[theorem]{Proposition}
\newtheorem{claim}[theorem]{Claim}
\newtheorem{example}[theorem]{Example}
\theoremstyle{remark}
\newtheorem{remark}[theorem]{Remark}
{\begin{equation} \begin{minipage}[b]{0.8\linewidth}}
{\end{minipage} \end{equation} \ignorespacesafterend}
\renewcommand{\l}{\left}
\renewcommand{\r}{\right}
\numberwithin{equation}{section}
\numberwithin{table}{theorem}
\title{Generalized Bernstein--Reznikov integrals

\renewcommand{\footnotemark}{}
\addtocounter{footnote}{-1}
}
\author{Jean--Louis Clerc, Toshiyuki Kobayashi, Bent \O rsted, Michael Pevzner}
\date{}
\begin{document}
\maketitle
\begin{abstract}
We find a closed formula for the triple integral on spheres in
$\mathbb{R}^{2n}\times\mathbb{R}^{2n}\times\mathbb{R}^{2n}$ whose
kernel is given by powers of the standard symplectic form. This
gives a new proof to the Bernstein--Reznikov integral formula in the
$n=1$ case. Our method also applies for linear and conformal
structures.
\end{abstract}

\section{Triple product integral formula}

We consider the symplectic form $[\ ,\ ]$ on $\mathbb{R}^{2n} =
\mathbb{R}^n \oplus \mathbb{R}^n$ given by
\begin{equation}\label{eqn:BJ}
[(x,\xi),(y,\eta)]
:= -\langle x,\eta\rangle + \langle y,\xi\rangle.
\end{equation}
In this paper we prove a closed formula for the following triple
integral:
\begin{theorem}\label{thm:gBZ}
Let $d\sigma$ be the Euclidean measure on the sphere $S^{2n-1}$.
Then,
\begin{align*}
&\int_{S^{2n-1}\times S^{2n-1}\times S^{2n-1}}
\bigl|[Y,Z]\bigr|^{\frac{\alpha-n}{2}}
\bigl|[Z,X]\bigr|^{\frac{\beta-n}{2}}
\bigl|[X,Y]\bigr|^{\frac{\gamma-n}{2}}
d\sigma(X) d\sigma(Y) d\sigma(Z)
\\
&=
\bigl(2\pi^{n-\frac{1}{2}}\bigr)^3
\frac{\Gamma\bigl(\frac{2-n+\alpha}{4}\bigr)
                               \Gamma\bigl(\frac{2-n+\beta}{4}\bigr)
                               \Gamma\bigl(\frac{2-n+\gamma}{4}\bigr)
                               \Gamma\bigl(\frac{\delta+n}{4}\bigr)}
     {\Gamma(n) \Gamma\bigl(\frac{n-\lambda_1}{2}\bigr)
                \Gamma\bigl(\frac{n-\lambda_2}{2}\bigr)
                \Gamma\bigl(\frac{n-\lambda_3}{2}\bigr)}
.
\end{align*}
Here,
$\alpha=\lambda_1-\lambda_2-\lambda_3$,
$\beta=-\lambda_1+\lambda_2-\lambda_3$,
$\gamma=-\lambda_1-\lambda_2+\lambda_3$,
$\delta=-\lambda_1-\lambda_2-\lambda_3 = \alpha + \beta + \gamma$.
\end{theorem}
The integral converges absolutely if and only if
$(\lambda_1,\lambda_2,\lambda_3)\in \mathbb{C}^3$
lies in the following
 non-empty open region
(see Proposition \ref{ex:symp})
 defined by:
\begin{align*}
&\operatorname{Re} \alpha > n-2,\,\,
 \operatorname{Re} \beta > n-2,\,\,
 \operatorname{Re} \gamma > n-2
&&(n \ge 2),
\\
&\operatorname{Re} \alpha > n-2,\,\,
 \operatorname{Re} \beta > n-2,\,\,
 \operatorname{Re} \gamma > n-2,\,\,
 \operatorname{Re} \delta > -1
\qquad
&&(n = 1).
\end{align*}
The integral under consideration extends as a meromorphic function
of $\lambda_1,\lambda_2$ and $\lambda_3$
 (\cite[Theorem 2]{xBG}, \cite{xsabbah}, see also \cite{G}). A
special case $(n=1)$ of Theorem \ref{thm:gBZ} was previously
established by J. Bernstein and A. Reznikov \cite{B2}.

The strategy of our proof is to interpret the triple product integral
as the trace of a certain integral operator,
for which we will find an explicit formula of eigenvalues and their
multiplicities.
For this,
our approach uses the Fourier transform in the ambient space,
appeals to the classical Bochner identity,
and finally reduces it to the special value of the hypergeometric
function ${}_5F_4$.
It gives a new proof even when $n=1$.

Sections \ref{sec:2} and \ref{sec:3} are devoted to the proof of
Theorem \ref{thm:gBZ}. In Section \ref{sec:4}, we discuss analogous
integrals of the triple product kernels involving $|x-y|^\lambda$ or
$|\langle x,y\rangle|^\lambda$ instead of
$\bigl|[x,y]\bigr|^\lambda$.

The underlying symmetries for
Theorem \ref{thm:gBZ} are given by the symplectic group
$Sp(n,\mathbb{R})$ of any rank $n$,
whereas those of Theorem \ref{thm:SO} correspond to
the rank one group $SO_0(m+1,1)$.
Even in rank one case,
 our methods give a new and
simple proof of the original results due to Deitmar \cite{D} for the case $|x-y|^\lambda$
(see Theorem \ref{thm:SO}).
Section \ref{sec:5} highlights some general
perspectives from representation theoretic point of view.

In Sections \ref{sec:2} and \ref{sec:Appendix} we have made an effort,
following  questions of the
referee, to explain the role of meromorphic families of homogeneous
distributions and the precise condition for the absolute convergence
of the triple integral, respectively.

Notations: $\mathbb N=\{0,1,2,\ldots\}$, $\mathbb R_+=\{x\in\mathbb
R:\,x>0\}$.

\section{Eigenvalues of integral transforms $\mathcal{T}_\mu$}
\label{sec:2}

We introduce a family of linear operators that depend
meromorphically on $\mu\in\mathbb C$ by
\[
\mathcal{T}_\mu:
C^\infty(S^{2n-1}) \to C^\infty(S^{2n-1})
\]
defined by
\begin{equation}\label{eqn:Tmu}
(\mathcal{T}_\mu f)(\eta)
:=
\int_{S^{2n-1}} f(\omega)
\bigl|[\omega,\eta]\bigr|^{-\mu-n} d\sigma(\omega).
\end{equation}

The integral \eqref{eqn:Tmu} converges absolutely if
$\operatorname{Re} \mu < -n+1$,
and has a meromorphic continuation for $\mu \in \mathbb{C}$.
If $\mu$ is real and sufficiently negative (e.g.\ $\mu < -n$),
then the kernel function
$K(\omega,\eta) := |[\omega,\eta]|^{-\mu-n}$
is square integrable on
$S^{2n-1} \times S^{2n-1}$
and
$K(\omega,\eta) = \overline{K(\eta,\omega)}$,
and consequently,
$\mathcal{T}_\mu$ becomes a self-adjoint, Hilbert--Schmidt operator on
$L^2(S^{2n-1})$.
In this section,
we determine all the eigenvalues of $\mathcal{T}_\mu$ and the
corresponding eigenspaces (see Theorem \ref{thm:Teigen}).

\subsection{Harmonic polynomials on $\mathbb{R}^{2n}$ and $\mathbb{C}^n$}

First, let us remind the classic theory of spherical
harmonics on real and complex vector spaces.

For $k\in\mathbb{N}$, we denote by $\mathcal{H}^k(\mathbb{R}^N)$ the
vector space consisting of homogeneous polynomials
$p(x_1,\dots,x_N)$ of degree $k$ such that
$\displaystyle\sum_{j=1}^N \frac{\partial^2}{\partial x_j^2} p = 0$.

In the polar coordinates
$x = r\omega$ ($r \ge 0$, $\omega \in S^{N-1}$),
we have
\begin{equation*}
\sum_{j=1}^N \frac{\partial^2}{\partial x_j^2}
= \frac{1}{r^2}
   \left( \Bigl(r\frac{\partial}{\partial r}\Bigr)^2
     + (N-2)\Bigl(r\frac{2}{\partial r}\Bigr)
     + \Delta_{S^{N-1}} \right),
\end{equation*}
where $\Delta_{S^{N-1}}$ denotes the Laplace--Beltrami operator on
the unit sphere endowed with the standard Riemannian metric.
In light that $r \frac{\partial}{\partial r} p = k p$
for a homogeneous function $p$ of degree $k$,
we see that the restriction $p|_{S^{N-1}}$ for
$p \in \mathcal{H}^k(\mathbb{R}^N)$ belongs to the following
eigenspace of $\Delta_{S^{N-1}}$:
\begin{equation}\label{eqn:eigenS}
V_k(S^{N-1})
:= \{ \varphi \in C^\infty (S^{N-1}) :
      \Delta_{S^{N-1}} \varphi = -k(k+N-2) \varphi \}.
\end{equation}

Since homogeneous functions on $\mathbb{R}^N$ are determined uniquely
by the restriction to $S^{N-1}$,
we get an injective map
\[
\mathcal{H}^k (\mathbb{R}^N) \to V_k(S^{N-1}),
\quad
p \mapsto p|_{S^{N-1}}
\]
for each $k \in \mathbb{N}$.
This map is also surjective (see \cite[Introduction, Theorem 3.1]{xhel} for example),
and we shall identify $\mathcal{H}^k(\mathbb{R}^N)$ with
$V_k(S^{N-1})$.
Thus, we regard the following algebraic direct sum
\[
\mathcal{H}(\mathbb{R}^N)
:= \bigoplus_{k=0}^\infty \mathcal{H}^k (\mathbb{R}^N)
\]
as a dense subspace of $C^\infty(S^{N-1})$.

Analogously, we can define the space of harmonic polynomials on
$\mathbb{C}^n$. For $\alpha,\beta \in \mathbb{N}$, we denote by
$\mathcal{H}^{\alpha,\beta}(\mathbb{C}^n)$ the vector space
consisting of polynomials $p(Z,\bar{Z})$ on $\mathbb{C}^n$ subject
to the following two conditions:
\begin{enumerate}[(1)]
\item  
$p(Z,\bar{Z})$ is homogeneous
of degree $\alpha$ in $Z = (z_1,\dots,z_n)$ and
of degree $\beta$ in $\bar{Z} = (\bar{z}_1,\dots,\bar{z}_n)$.
\item  
$\displaystyle\sum_{j=1}^n \dfrac{\partial^2}{\partial
z_i\partial\bar{z}_i}
 p(Z,\bar{Z}) = 0$.
\end{enumerate}
Then, $\mathcal{H}^{\alpha,\beta}(\mathbb{C}^n)$ is a finite
dimensional vector space.
It is non-zero except for the case where $n=1$ and $\alpha,\beta \ge 1$.

By definition, we have a natural linear isomorphism:
\begin{equation}\label{eqn:Habk}
\mathcal{H}^k(\mathbb{R}^{2n})
\simeq \bigoplus_{\alpha+\beta=k}
\mathcal{H}^{\alpha,\beta} (\mathbb{C}^n).
\end{equation}

We shall see that $\mathcal{H}^{\alpha,\beta}(\mathbb{C}^n)$ is an
eigenspace of the operator $\mathcal{T}_\mu$ for any $\mu$ and for
every $\alpha$ and $\beta$. To be more precise, we introduce a
meromorphic function of $\mu$ by
\begin{equation}\label{eqn:Ak}
A_k(\mu,\mathbb{C}^n)
\equiv
A_k(\mu)
:= \begin{cases}
      0
      &(k: \text{odd}),
   \\[1ex]
      2\pi^{n-\frac{1}{2}}
      \dfrac{\Gamma(\frac{1-n-\mu}{2})\Gamma(\frac{k+n+\mu}{2})}
            {\Gamma(\frac{n+\mu}{2})\Gamma(\frac{k+n-\mu}{2})}
      &(k: \text{even}).
   \end{cases}
\end{equation}

We shall use the notation $A_k(\mu,\mathbb{C}^n)$ when we emphasize
the ambient space $\mathbb C^n$ (see (\ref{eqn:RCN})).

\begin{theorem}\label{thm:Teigen}
For $\alpha,\beta \in \mathbb{N}$,
\[
\mathcal{T}_\mu \Big|_{\mathcal{H}^{\alpha,\beta}(\mathbb{C}^n)}
= (-1)^\beta A_{\alpha+\beta}(\mu) \operatorname{id}.
\]
\end{theorem}
The rest of this section is devoted to the
 proof of Theorem \ref{thm:Teigen}.

\subsection{Preliminary results on homogeneous distributions}

In this section we collect some basic concepts and results on
distributions in a way that we shall use later.
See \cite[Chapters 1 and 2]{G}, and also \cite[Appendix]{KM-intopq}.

A distribution $F_\nu$ depending on a complex parameter $\nu$ is
defined to be meromorphic if for every test function $\varphi$,
$\langle F_\nu,\varphi\rangle$ is a meromorphic function in $\nu$.
We say $F_\nu$ has a pole at $\nu=\nu_0$ if $\langle
F_\nu,\varphi\rangle$ has a pole at $\nu=\nu_0$ for some $\varphi$.
Then, taking its residue at $\nu_0$, we get a new distribution
\[
\varphi \mapsto \underset{\nu=\nu_0}{\operatorname{res}}
\langle F_\nu, \varphi \rangle,
\]
which we denote by
$\underset{\nu=\nu_0}{\operatorname{res}} F_\nu$.

Suppose $F$ is a distribution defined in a conic open subset in
$\mathbb{R}^N$.
We say $F$ is homogeneous of degree $\lambda$ if
\begin{equation}\label{eqn:euler}
\sum_{j=1}^N x_j \frac{\partial}{\partial x_j} F = \lambda F
\end{equation}
in the sense of distributions, or equivalently, \( \langle F,
\varphi \left(\frac{1}{a} \cdot\right) \rangle = a^{\lambda+N}
\langle F,\varphi\rangle \) for any test function $\varphi$ and
$a>0$.

Globally defined homogeneous distributions on $\mathbb{R}^N$ are
determined by their restrictions to $\mathbb{R}^N \setminus \{0\}$ for
generic degree:
\begin{lemma}\label{lem:homo}
Suppose $f$ is a distribution on $\mathbb{R}^N$ which is homogeneous
of degree $\lambda$. If $f|_{\mathbb{R}^N \setminus \{0\}} = 0$ and
$\lambda \notin \{-N,-N-1,-N-2,\dotsc\}$, then $f=0$ as distribution
on $\mathbb{R}^N$.
\end{lemma}

\begin{proof}
By the general structural theory on distributions, if
$\operatorname{supp} f \subset \{0\}$ then $f$ must be a finite
linear combination of the Dirac delta function $\delta(x)$ and its
derivatives. On the other hand, the degree of the delta function and
its derivatives is one of $-N, -N-1, -N-2, \dotsc$. By our
assumption on $f$, this does not happen. Hence, we conclude $f=0$ as
distribution on $\mathbb{R}^N$.
\end{proof}

For a given function $p \in C^\infty (S^{N-1})$,
we define its extension into a homogeneous function of degree $\lambda$ by
\[
p_\lambda(r\omega)
:= r^\lambda p(\omega),
\quad
(r>0, \  \omega \in S^{N-1}).
\]
We regard the locally integrable functions as distributions by
multiplying the Lebesgue measure.

\begin{lemma}\label{lem:plmd}
Let $p\in C^\infty(S^{N-1})$, then
\begin{enumerate}[\upshape 1)]
\item  
$p_\lambda$ is locally integrable on $\mathbb{R}^N$ if\/
$\operatorname{Re}\lambda > -N$.
\item  
$p_\lambda$ extends to a tempered distribution which depends
meromorphically on $\lambda \in \mathbb{C}$. Its poles are
simple and contained in the set $\{-N, -N-1, \dotsc\}$.
\item  
The distribution $p_\lambda$ is homogeneous of degree $\lambda$ in the sense
of (\ref{eqn:euler}) if $\lambda$ is not a pole.
\end{enumerate}
\end{lemma}

\begin{proof}
1)
Clear from the formula of the Lebesgue measure
$dx = r^{N-1} dr d\sigma(\omega)$
in the polar coordinates $x = r\omega$ ($r>0$, $\omega \in S^{N-1}$).

2)
Take a test function
$\varphi \in \mathcal{S}(\mathbb{R}^N)$.
Suppose first $\operatorname{Re}\lambda>-N$.
Then, we can decompose $\langle p_\lambda,\varphi\rangle$ as
\[
\langle p_\lambda, \varphi \rangle
=
\int_{|x|\le1} p_\lambda(x) \varphi(x)dx
+ \int_{|x|>1} p_\lambda(x) \varphi(x)dx.
\]
The second term extends holomorphically in the entire complex plane.
Let us prove that the first term extends meromorphically in
$\mathbb{C}$. For this, we fix $k \in \mathbb{N}$, and consider the
Taylor expansion of $\varphi$:
\[
\varphi(x)
= \sum_{|\alpha|\le k} \frac{\varphi^{(\alpha)}(0)}{\alpha!} x^\alpha
  + \varphi_k(x),
\]
where $\alpha = (\alpha_1,\dots,\alpha_N)$ is a multi-index,
$x^\alpha = x_1^{\alpha_1} \cdots x_N^{\alpha_N}$,
$|\alpha| = \alpha_1+\dots+\alpha_N$,
and $\varphi_k(x) = O (|x|^{k+1})$.
Accordingly, we have
\begin{align*}
&\int_{|x|\le1} p_\lambda(x) \varphi(x) dx
\\
={}&
\sum_{|\alpha|\le k} \frac{\varphi^{(\alpha)}(0)}{\alpha!}
\int_{S^{N-1}} p(\omega) \omega^\alpha d\sigma(\omega)
\int_0^1 r^{\lambda+|\alpha|+N-1} dr
+ \int_{|x|\le1} p_\lambda(x) \varphi_k(x) dx
\\
={}&
\sum_{|\alpha|\le k} \frac{1}{\lambda+|\alpha|+N}
\frac{\varphi^{(\alpha)}(0)}{\alpha!}
\int_{S^{N-1}} p(\omega) \omega^\alpha d\sigma(\omega)
+ \int_{|x|\le1} p_\lambda \varphi_k(x) dx.
\end{align*}
The last term extends holomorphically to the open set
$\{\lambda\in\mathbb{C}: \operatorname{Re}\lambda>-N-k\}$. Since $k$
is arbitrary we see that $\langle p_\lambda, \varphi\rangle$ extends
meromorphically to the entire complex plane, and all its poles are
simple and contained in the set $\{-N,-N-1,-N-2,\dotsc\}$. Thus, the
second statement is proved.

3) The differential  equation $\displaystyle\sum_{j=1}^N x_j
\frac{\partial}{\partial x_j} p_\lambda(x) = \lambda p_\lambda(x)$
holds in the sense of distributions for $\operatorname{Re}\lambda
\gg 0$. This equation extends to all complex $\lambda$ except for
poles because the distribution $p_\lambda$ depends meromorphically
on
$\lambda$. 
\end{proof}

\begin{example}[$k=0$ case]\label{ex:k0}
For $k=0$, $\mathcal{H}^k(\mathbb{R}^N)$ is one-dimensional, spanned
by the constant function $\mathbf{1}$. We denote by $r^\lambda$ the
corresponding homogeneous distribution $\mathbf{1}_\lambda$.

As we saw in the proof of Lemma \ref{lem:plmd},
the distribution $r^\lambda$ has a simple pole at $\lambda=-N$ and its
residue is given by
\begin{equation}\label{eqn:resr}
\underset{\lambda=-N}{\operatorname{res}} r^\lambda
= \operatorname{vol}(S^{N-1}) \delta(x)
= \frac{2\pi^{\frac{N}{2}}}{\Gamma(\frac{N}{2})} \delta(x).
\end{equation}
\end{example}

\begin{example}[$N=1$ case]\label{ex:Riesz}
In the one dimensional case,
$S^{N-1}$ consists of two points,
$1$ and $-1$,
and consequently,
the homogeneous distribution $p_\lambda$ is determined by the values
$p_\lambda(1)$ and $p_\lambda(-1)$.
{}From this viewpoint,
we give a list of classical homogeneous distributions on $\mathbb{R}$.
\begin{table}[H]
$$
\begin{array}{c|cccccc}
 p_\lambda
 &\quad  x_+^\lambda  \quad
 &\quad  x_-^\lambda  \quad
 &\quad  |x|^\lambda  \quad
 &|x|^\lambda\operatorname{sgn}x
 &(x+i0)^\lambda
 &(x-i0)^\lambda
\\ &
\\
\hline
\\
p(1) &1 &0 &1 &1 &1 &1
\\ &
\\
p(-1) &0 &1 &1 &-1 &e^{i\pi\lambda} &e^{-i\pi\lambda}
\end{array}
$$
\caption{Homogeneous distributions on $\mathbb{R}$}
\label{table:2.2}
\end{table}
\noindent
The notation $(x\pm i0)^\lambda$ indicates that these distributions are
obtained as the boundary values of holomorphic functions in the upper
(or lower) half plane.
For $\operatorname{Re}\lambda>-1$,
\[
\lim_{\varepsilon\downarrow 0} (x\pm i\varepsilon)^\lambda
= (x\pm i0)^\lambda
\]
holds both in the ordinary sense and in distribution sense.
The distributions $(x\pm i0)^\lambda$ extend holomorphically to all
complex $\lambda$,
whereas the poles of $x_\pm^\lambda$, $|x|^\lambda$,
 $|x|^\lambda \operatorname{sgn}x$ are located at
$\{-1,-2,-3,\dotsc\}$, $\{-1,-3,-5,\dotsc\}$, $\{-2,-4,-6,\dotsc\}$,
respectively.

For $\lambda \ne -1,-2,\dotsc$, any two in Table \ref{table:2.2}
form a basis in the space of homogeneous distributions of degree
$\lambda$. For example, by a simple basis change one gets:
\begin{equation}\label{eqn:xi0}
(x-i0)^\lambda
= e^{-i\frac{\pi\lambda}{2}}
  (\cos\frac{\pi}{2}\lambda |x|^\lambda
   + i \sin\frac{\pi}{2}\lambda |x|^\lambda \operatorname{sgn} x ).
\end{equation}
\end{example}

\subsection{Application of the Bochner identity}

Let $\langle \ ,\ \rangle$ be the standard inner product on
$\mathbb{R}^N$. We consider the Fourier transform $\mathcal{F}
\equiv \mathcal{F}_{\mathbb{R}^N}$ on $\mathbb{R}^N$ normalized by
\[
(\mathcal{F} f)(Y)
:= \int_{\mathbb{R}^N} f(X)
   e^{-2\pi i \langle X,Y\rangle} dX,
\]
and  we extend $\mathcal{F}$ to the space
$\mathcal{S}'(\mathbb{R}^N)$ of tempered distributions.

If $f \in \mathcal{S}'(\mathbb{R}^N)$ is homogeneous of degree
$\lambda$,
then its Fourier transform $\mathcal{F}f$ is homogeneous of degree
$-\lambda-N$.

\begin{example}[$N=1$ case]\label{ex:F}
~
\begin{enumerate}[\upshape 1)]
\item  
$\displaystyle
\mathcal{F}(x_+^\lambda)(y)
= \frac{e^{-\frac{i\pi}{2}(\lambda+1)} \Gamma(\lambda+1)}{(2\pi)^{\lambda+1}}
  (y-i0)^{-\lambda-1}
$.
\item  
$\displaystyle
\mathcal{F}(|x|^\lambda)(y)
= \frac{\Gamma(\frac{\lambda+1}{2})}
       {\pi^{\lambda+\frac{1}{2}}\Gamma(\frac{-\lambda}{2})}
  |y|^{-\lambda-1} $, and

$\displaystyle
\mathcal{F}(|x|^\lambda \operatorname{sgn}x)(y)
= \frac{-i\Gamma(\frac{\lambda+2}{2})}
       {\pi^{\lambda+\frac{1}{2}} \Gamma(\frac{1-\lambda}{2})}
  |y|^{-\lambda-1} \operatorname{sgn} y
$.
\end{enumerate}
\end{example}
These formulas may be found for  instance in \ \cite[Chapter II,
\S2.3]{G}, however, we shall give a brief proof because its
intermediate step (e.g.\ \eqref{eqn:Fexlmd} below) will be used
later (see the proof of Proposition \ref{prop:QF}).

\begin{proof}[Proof of Example \ref{ex:F}]
1)
Suppose $\operatorname{Re}\lambda > -1$.
Then $x_+^\lambda$ is locally integrable on $\mathbb{R}$,
and we have
\[
\lim_{\varepsilon\downarrow0}
e^{-2\pi\varepsilon x} x_+^\lambda
= x_+^\lambda
\]
both in the ordinary sense and in the sense of distributions. Then,
by Cauchy's integral formula and by the definition of the Gamma
function, we get
\begin{equation}\label{eqn:Fexlmd}
\mathcal{F}(e^{-2\pi\varepsilon x} x_+^\lambda)(y)
= \frac{e^{-\frac{\pi i}{2}(\lambda+1)}\Gamma(\lambda+1)}
          {(2\pi)^{\lambda+1}}
  (y-i\varepsilon)^{-\lambda-1},
\end{equation}
for $\varepsilon>0$.
Taking the limit as $\varepsilon\to 0$ we get the desired identity for
$\operatorname{Re}\lambda > -1$.

By the meromorphic continuation on $\lambda$,
the first statement is proved.

2) Similarly to 1), we can obtain a closed formula for
$\mathcal{F}(x_-^\lambda)(y)$. Then the second statement follows
readily from the base change matrix for the three bases
$\{x_+^\lambda,x_-^\lambda\}$, $\{|x|^\lambda,|x|^\lambda
\operatorname{sgn}x \}$, and $\{(x+i0)^\lambda, (x-i0)^\lambda \}$
for homogeneous distributions on $\mathbb{R}$. (We also use the
duplication formula of the Gamma function.)
\end{proof}

We are ready to state the main result of this subsection.
Let us define the following meromorphic function of $\lambda$ by
\[
B_N(\lambda,k)
:= \pi^{-\lambda-\frac{N}{2}} i^{-k}
   \frac{\Gamma(\frac{k+\lambda+N}{2})}{\Gamma(\frac{k-\lambda}{2})}.
\]

\begin{lemma}\label{lem:Boch}
For any $p \in \mathcal{H}^k(\mathbb{R}^N)$,
we have the following identity
\begin{equation}\label{eqn:Fplmd}
\mathcal{F} p_\lambda
= B_N (\lambda,k) p_{-\lambda-N}
\end{equation}
as distributions on $\mathbb{R}^N$ that depend meromorphically on $\lambda$.
\end{lemma}

\begin{example}\label{ex:Fdelta}
Since \eqref{eqn:Fplmd} is an identity for meromorphic
distributions, we can pass to the limit, or compute residues at
special values whenever it makes sense. For instance, let $k=0$.
Then, by \eqref{eqn:resr},
 the special value of \eqref{eqn:Fplmd} at
$\lambda=0$ yields
\[
\mathcal{F}(\mathbf{1})
= \lim_{\lambda\to0} B_N(\lambda,0) r^{-\lambda-N}
= \delta(y).
\]
In view of the identity
\(
B_N(\lambda,0) B_N(-\lambda-N,0) = 1,
\)
the residue of \eqref{eqn:Fplmd} at $\lambda=-N$ yields
\[
\mathcal{F}(\delta(x)) = \mathbf{1}.
\]
This, of course, is in agreement with the inversion formula for the
Fourier transform.
\end{example}

\begin{proof}[Proof of Lemma \ref{lem:Boch}]
For $N=1$, $k$ equals either $0$ or $1$, and correspondingly,
$p_\lambda$ is a scalar multiple of $|x|^\lambda$ or $|x|^\lambda
\operatorname{sgn} x$, respectively. Hence, Lemma \ref{lem:Boch} in
the case $N=1$ is equivalent to Example \ref{ex:F} 2).

Let us prove \eqref{eqn:Fplmd} for $N \ge 2$ as the identity of
distributions on $\mathbb{R}^N$.
We shall first prove the identity \eqref{eqn:Fplmd} on
$\mathbb{R}^N \setminus \{0\}$ in the non-empty domain:
\begin{equation}\label{eqn:lmdstrip}
-N < \operatorname{Re} \lambda < -\frac{1}{2}(N+1).
\end{equation}
Since the both sides of \eqref{eqn:Fplmd} are homogeneous
distributions of the same degree, this will imply that the identity
\eqref{eqn:Fplmd} holds on $\mathbb{R}^N$ by Lemma \ref{lem:homo}.
Further, since the both sides of \eqref{eqn:Fplmd} depend
meromorphically on $\lambda$ by Lemma \ref{lem:plmd}, the identity
\eqref{eqn:Fplmd} holds for all $\lambda$ in the sense of
distributions that depend meromorphically on $\lambda$.

The rest of this proof is devoted to show \eqref{eqn:Fplmd} on
$\mathbb{R}^N \setminus \{0\}$ in the domain \eqref{eqn:lmdstrip}.
For this, it is sufficient to prove that
\[
\langle \mathcal{F}p_\lambda, gq \rangle
= B_N(\lambda,k) \langle p_{-\lambda-N}, gq \rangle,
\]
for any compactly supported function $g\in C_c^\infty(\mathbb{R}_+)$
and any $q \in \mathcal{H}^l(\mathbb{R}^N)$ $(l \in \mathbb{N})$
because the linear spans of such functions form a dense subspace in
$C_c^\infty(\mathbb{R} \setminus \{0\})$. Here, $gq$ stands for a
function on $\mathbb{R}^N \setminus \{0\}$ defined by
\[
(gq)(s\eta) = g(s)q(\eta)
\quad
(s>0, \eta \in S^{N-1}).
\]

By definition of the Fourier transform on
$\mathcal{S}'(\mathbb{R}^N)$, $\langle
\mathcal{F}{p_\lambda},gq\rangle=\langle
p_\lambda,\mathcal{F}(gq)\rangle$. Hence, what we need to prove is
\begin{equation}\label{eqn:Fpq}
\langle p_\lambda,\mathcal{F}(gq) \rangle
= B_N (\lambda,k) \langle p_{-\lambda-N}, gq \rangle.
\end{equation}
We note that both $p_\lambda$ and $p_{-\lambda-N}$ are locally
integrable functions on $\mathbb{R}^N$ under the assumption
\eqref{eqn:lmdstrip}.
To calculate the left-hand side of \eqref{eqn:Fpq},
we use the  Bochner identity for
$q \in \mathcal{H}^l (\mathbb{R}^N)$:
\[
\int_{S^{N-1}} q(\omega) e^{-i\nu \langle \omega,\eta\rangle} d\omega
= (2\pi)^{\frac{N}{2}} i^{-l} \nu^{1-\frac{N}{2}}
  J_{l+\frac{N}{2}-1} (\nu) q(\eta),
\]
where $J_\mu(\nu)$ denotes the Bessel function of the first kind.
Then, we get the following formula after a change of variables
$x=2\pi rs$:
\[
\mathcal{F}(gq)(r\omega)
= 2\pi i^{-l} r^{1-\frac{N}{2}} q(\omega)
\int_0^\infty s^{\frac{N}{2}} g(s)
J_{l+\frac{N}{2}-1} (2\pi rs)ds.
\]
Hence, we have
\begin{equation}\label{eqn:Fubini}
\langle p_\lambda, \mathcal{F}(gq) \rangle
= \int_0^\infty \int_{S^{N-1}}
\Bigl(\int_0^\infty I(r,s)ds\Bigr)
p(\omega) q(\omega) d\sigma(\omega) dr,
\end{equation}
where
we set
\[
I(r,s)
:= 2\pi i^{-l} r^{\lambda+\frac{N}{2}} s^{\frac{N}{2}} g(s)
J_{l+\frac{N}{2}-1} (2\pi rs).
\]
At this point,
we prepare the following:
\begin{claim}\label{claim:Irs}
Assume that
$\lambda$ satisfies \eqref{eqn:lmdstrip}
and that $g$ is compactly supported in $\mathbb{R}_+$.
\begin{enumerate}[\upshape 1)]
\item  
$I(r,s) \in L^1 (\mathbb{R}_+ \times \mathbb{R}_+, drds)$.
\item  
$\displaystyle
\int_0^\infty I(r,s) ds
= B_N(\lambda,l) g(s) s^{-\lambda-1}$.
\end{enumerate}
\end{claim}

\begin{proof}[Proof of Claim \ref{claim:Irs}]
1)
Since the support of $g$ is away from $0$ and $\infty$,
it follows from the asymptotic behaviour of the Bessel function
$J_\mu(z)$ as $z\to 0$ and $z\to \infty$ that
there exists a constant $c>0$ such that
\[
|I(r,s)| \le
\begin{cases}
   c \ r^{\operatorname{Re}\lambda+N+l-1}
   &\text{as $r \to 0$},
\\
   c \ r^{\operatorname{Re}\lambda+\frac{1}{2}(N-1)}
   &\text{as $r \to \infty$}.
\end{cases}
\]
By the assumption
$-N < \operatorname{Re}\lambda < -\frac{1}{2}(N+1)$,
we conclude $I(r,s)$ is an integrable function on
$\mathbb{R}_+ \times \mathbb{R}_+$.

2)
This is a direct consequence of the
 following classical formula of the Hankel transform
\cite[6.561.14]{GR}
\[
\int_0^\infty x^\mu J_\nu(x) dx
= 2^\mu \frac{\Gamma(\frac{1+\nu+\mu}{2})}{\Gamma(\frac{1+\nu-\mu}{2})},
\]
for $\operatorname{Re}(\mu+\nu) > -1$ and
 $\operatorname{Re}\mu < -\frac{1}{2}$.
\end{proof}

Returning to the proof of Lemma \ref{lem:Boch},
we can now apply Fubini's theorem for
the right-hand side of \eqref{eqn:Fubini}
to get
\begin{align*}
\langle p_\lambda, \mathcal{F}(gq) \rangle
&=
\Bigl(\int_{S^{N-1}} p(\omega)q(\omega)d\sigma(\omega)\Bigr)
\int_0^\infty \Bigl(\int_0^\infty I(r,s)ds\Bigr) dr
\\
&=
B_N(\lambda, l)
\int_{S^{N-1}} p(\omega) q(\omega)d\sigma(\omega)
\int_0^\infty g(s) s^{-\lambda-1} ds.
\\
\intertext{We recall that $ p\in\mathcal H^k(\mathbb R^N)$ and $q\in\mathcal H^l(\mathbb R^N)$,
therefore the first factor is non-zero only if $k=l$.
Then the right-hand side equals}
&=
B_N(\lambda,k)
\langle p_{-\lambda-N}, gq \rangle.
\end{align*}
Hence \eqref{eqn:Fpq} is proved in the non-empty open domain of
$\lambda$ satisfying the inequality \eqref{eqn:lmdstrip}. Therefore,
the proof of Lemma \ref{lem:Boch} is completed.
\end{proof}

\subsection{Fourier transform of homogeneous functions}\label{sec:Frest}

We consider the restriction of the
Fourier transform $\mathcal{F}$ on $\mathbb{R}^N$ to the space of
homogeneous functions.
For $\mu \in \mathbb{C}$,
we set
\begin{eqnarray}\label{eqn:Vmu}
V_\mu&\equiv& V_\mu(\mathbb R^N)\\
&:=& \{ f\in C^\infty(\mathbb{R}^N \setminus \{0\}): f(tX)
= |t|^{-\mu-\frac{N}{2}} f(X)
\,\text{for any $t \in \mathbb{R} \setminus \{0\}$} \}.\nonumber
\end{eqnarray}
Then, $V_\mu$ may be regarded as a subspace of the space
$\mathcal{S}'(\mathbb{R}^N)$ of tempered distributions for
$\mu\neq\frac N2,\frac N2+2,\dotsc$.
We note that $f \in V_\mu$ is determined by the restriction
$f|_{S^{N-1}}$,
and thus $V_\mu$ can be identified with the space of smooth even
functions on $S^{N-1}$.
In this subsection,
we will prove:
\begin{proposition}\label{prop:FVV}
Suppose $|\mu| \ne \frac{N}{2}, \frac{N}{2}+2, \dotsc$.
Then
the Fourier transform $\mathcal{F}:\mathcal S'(\mathbb R^N) \to \mathcal S'(\mathbb R^N)$ induces a bijection between $V_{-\mu}$ and
$V_\mu$.
\end{proposition}

For the proof of this proposition,
we prepare some general result as follows.
Let $M$ be a compact smooth Riemannian manifold.
We write $\Delta_M$ for the Laplace--Beltrami operator,
$\operatorname{Ker}(\Delta_M-\lambda)$ for the eigenspace
$\{ f \in C^\infty(M) : \Delta_M f = \lambda f \}$,
and $\tau$ for the Riemannian volume element.
Then we can regard $C^\infty(M)$ as a
subspace of $\mathcal{D}'(M)$,
the space of distributions,
by $f(x) \mapsto f(x) d\tau(x)$.

\begin{lemma}\label{lem:ACD}
Suppose
$A : C^\infty(M) \to \mathcal{D}'(M)$
is a linear map satisfying the following two properties:

\begin{equation}
\label{eqn:ACD1}
\hbox{$A$ acts as a scalar,
say $a(\lambda) \in \mathbb{C}$,
on each eigenspace
$\operatorname{Ker}(\Delta_M-\lambda)$.}
\end{equation}
\begin{align}
\label{eqn:ACD2}
&\hbox{
$a(\lambda)$ is at most of polynomial growth, namely,}
\notag
\\
&\hbox{
there exist $C, N > 0$ such that
$|a(\lambda)| \le C(1+|\lambda|)^N$.}
\end{align}

Then, $A\psi \in C^\infty(M)$ for any $\psi \in C^\infty(M)$.
\end{lemma}

\begin{remark}\label{rem:ACD}
For a real analytic manifold $M$,
an analogous statement holds for a linear map
$A : \mathcal{A}(M) \to \mathcal{B}(M)$,
where $\mathcal{A}$, $\mathcal{B}$ denote the sheaf of real analytic functions,
(Sato's) hyperfunctions, respectively, if $a(\lambda)$ is at most of
infra exponential growth (see \cite[Section 2.3]{xAnn}).
\end{remark}

\begin{proof}[Proof of Lemma \ref{lem:ACD}]
By Sobolev's lemma, a distribution $T$ on $M$ belongs to
$C^\infty(M)$ if and only if $\Delta_M^l T$ (in the sense of
distributions) belongs to $L^2(M)$ for any $l \in \mathbb{N}$. We
will show that this is the case for $T = A \psi$ if $\psi \in
C^\infty(M)$.

Let $0=\lambda_0 > \lambda_1\geq \lambda_2\geq \dots$ the (negative) eigenvalues of $\Delta_M$,
 repeated according to their multiplicites.
We take an orthonormal basis $\{\varphi_j : j= 0,1,2,\dots\}$ in $L^2(M)$
 consisting of real-valued eigenfunctions of $\Delta_M$ with eigenvalues $\lambda_j$.
Then any distribution $T$ on $M$ can be expanded into a series of eigenfunctions (as a distribution):
\[
     T=\sum_j c_j \varphi_j, \quad c_j : = \langle T, \varphi_j\rangle\ .
\]
For $l\in \mathbb N$, the condition $\Delta_M^l T \in L^ 2(M)$  amounts to
\[
      \sum_j \vert c_j\vert^2 \lambda_j^{2l} < \infty \ ,
\]
so that
\[
      T \in C^\infty(M) \Longleftrightarrow \sum_j \vert c_j\vert^2(1+\vert \lambda_j\vert)^{2l}<\infty \quad {\rm for\ any\ } l \in \mathbb N\ .
\]
Take any $\psi \in C^\infty(M)$, and expand it into a series of eigenfunctions
\[
       \psi = \sum_j b_j \varphi_j, \quad b_j = \langle \psi, \varphi_j\rangle_{L^2(M)}\ .
\]

Applying the operator $A$, we get from \eqref{eqn:ACD1}
\[ A\psi = \sum_j b_j a(\lambda_j) \varphi_j\ .
\]
For any $l\in \mathbb N$, using \eqref{eqn:ACD2}
\[\sum_j \vert b_j a(\lambda_j)\vert^2 \lambda_j^{2l}\leq C \sum_j
\vert b_j\vert^2 (1+\vert \lambda_j\vert)^{2N} \lambda_j^{2l}\leq C \sum_j \vert b_j\vert^2 (1+\vert \lambda_j)\vert^{2N+2l} <+\infty\ .\]
Thus $A\psi \in C^\infty(M)$.
Hence Lemma \ref{lem:ACD} is proved.
\end{proof}

We are ready to complete the proof of Proposition \ref{prop:FVV}.
\begin{proof}[Proof of Proposition \ref{prop:FVV}]
For $f \in V_{-\mu}$,
$\mathcal{F}f$ is a homogeneous distribution of degree
$-\mu - \frac{N}{2}$.
Therefore, to see that
$\mathcal{F}f \in C^\infty (\mathbb{R}^N \setminus \{0\})$,
it is sufficient to show that the restriction
$\mathcal{F}f |_{S^{N-1}}$
is a smooth function on $S^{N-1}$.
This follows from the general result
(see Lemma \ref{lem:ACD}) together with Lemma \ref{lem:Boch} and
Stirling's formula on the asymptotic behaviour of the Gamma function.
Hence, Proposition \ref{prop:FVV} is proved.
\end{proof}

\subsection{Operator $\mathcal{T}_\mu$ and Symplectic Fourier transform
$\mathcal{F}_J$}
\label{subsec:2.4}

The key idea to find eigenvalues of the integral transform
$\mathcal{T}_\mu$ on $L^2(S^{2n-1})$ is to interpret it as the
restriction of the \textit{symplectic Fourier transform}, to be
denoted by $\mathcal{F}_J$,
 on the ambient space
$\mathbb{R}^{2n}$.

If $\operatorname{Re} \mu < -\frac{N}{2}+1$,
the following integral converges absolutely for any
$h \in C^\infty(S^{N-1})$:
\begin{equation}\label{eqn:Qmu}
(\mathcal{Q}_\mu h)(\eta)
:= \int_{S^{N-1}} |\langle\omega,\eta\rangle|^{-\mu-\frac{N}{2}}
   h(\omega)d\omega.
\end{equation}
Then $Q_\mu h$ extends meromorphically on $\mu \in \mathbb{C}$,
whose poles are simple and contained in the set
$\{ 1-\frac{N}{2}, 3-\frac{N}{2}, 5-\frac{N}{2}, \dotsc \}$.
Thus,
we get a family of linear operators
that depend meromorphically on $\mu\in\mathbb C$ by
\[
\mathcal{Q}_\mu :
C^\infty(S^{N-1}) \to C^\infty(S^{N-1}).
\]
We may regard $Q_\mu h$ as an even homogeneous function on
$\mathbb{R}^N \setminus \{0\}$ of degree $-\mu-\frac{N}{2}$ by simply
letting $\eta$ be a variable in $\mathbb{R}^N \setminus \{0\}$.
Then,
$Q_\mu h \in V_\mu$.
By Proposition \ref{prop:FVV},
 the Fourier transform
$\mathcal{F}$ gives a bijection between $V_{-\mu}$ and $V_{\mu}$
for $|\mu| \ne \frac{N}{2}, \frac{N}{2}+2, \dotsc$.
On
the other hand, $V_\mu$ can be identified with the space of smooth
even functions on $S^{N-1}$. We notice that the latter space is
independent of $\mu$. Thus, we have the following diagram:
\[
\begin{array}{rcccc}
   \mathcal{F} & :
      &\mathcal{S}'(\mathbb{R}^N) & \stackrel{\sim}{\longrightarrow} & \mathcal{S}'(\mathbb{R}^N)
\\[1ex]
      && \cup   & \rotatebox[origin=c]{180}{$\circlearrowright$} & \cup
\\[1ex]
      && V_{-\mu}  & \stackrel{\sim}{\longrightarrow} & V_{\mu}
\\[1ex]
      && \cap   && \cap
\\[1ex]
   \mathcal{Q}_\mu & :
      & C^\infty(S^{N-1}) & \longrightarrow & C^\infty(S^{N-1})
\end{array}
\]
The lower diagram commutes up
to a scalar constant.
To make a precise statement,
we define
\begin{align}\label{eqn:CNmu}
C_N(\mu) :={}& \frac{(2\pi)^{\mu+\frac{N}{2}}}
                    {\Gamma(\mu+\frac{N}{2})\cos\frac{\pi}{2}(\mu+\frac{N}{2})}
\nonumber
\\
={}&
\frac{2\pi^{\mu+\frac{N-1}{2}} \Gamma(\frac{2-N-2\mu}{4})}
     {\Gamma(\frac{N+2\mu}{4})}.
\end{align}
Then we have:
\begin{proposition}\label{prop:QF} As operators that depend
meromorphically on $\mu$, $\mathcal Q_\mu$ satisfy the following
identity:
\[
\mathcal{Q}_\mu
= C_N(\mu) \mathcal{F} \big|_{V_{-\mu}}.
\]
\end{proposition}

\begin{proof}
Any element in $V_{-\mu}$ is of the form
\[
h_{\mu-\frac{N}{2}}(r\omega)
= r^{\mu-\frac{N}{2}} h(\omega)
\quad
(r>0, \ \omega\in S^{N-1}),
\]
for some $h \in C^\infty(S^{N-1})$ which is an even function, i.e.,
$h(\omega) = h(-\omega)$.

We shall prove
\begin{equation}\label{eqn:QhF}
Q_\mu h_{\mu-\frac{N}{2}}
= C_N(\mu) \mathcal{F} h_{\mu-\frac{N}{2}}
\end{equation}
as distributions on $\mathbb{R}^N \setminus \{0\}$.
For each fixed $h$,
the both sides of \eqref{eqn:QhF} are distributions that depend
meromorphically on $\mu$.
Therefore, it is sufficient to prove \eqref{eqn:QhF} for some
non-empty open
domain in $\mu$, say,
\begin{equation}\label{eqn:muN2}
\operatorname{Re}\mu > -\frac{N}{2}.
\end{equation}
The inequality \eqref{eqn:muN2} implies that $h_{\mu-\frac{N}{2}}
\in L_{loc}^1 (\mathbb{R}^N)$, and we have
\[
\lim_{\varepsilon\downarrow0} e^{-2\pi\varepsilon r}
h_{\mu-\frac{N}{2}} (r\omega)
= h_{\mu-\frac{N}{2}} (r\omega)
\]
as a locally integrable function,
and also in $\mathcal{S}'(\mathbb{R}^N)$.
Hence, taking the Fourier transform, we get
\[
\lim_{\varepsilon\downarrow0} \mathcal{F} (e^{-2\pi\varepsilon r}
h_{\mu-\frac{N}{2}})
= \mathcal{F} h_{\mu-\frac{N}{2}}
\]
in $\mathcal{S}'(\mathbb{R}^N)$.

Let us compute $\mathcal{F}(e^{-2\pi\varepsilon r}
h_{\mu-\frac{N}{2}})$. Below, we use the Fourier transform for both
$\mathbb{R}^N$ and $\mathbb{R}$, which will be denoted by
$\mathcal{F}_{\mathbb{R}^N}$ and $\mathcal{F}_{\mathbb{R}}$ to avoid
confusion. We note that $e^{-2\pi\varepsilon r} h_{\mu-\frac{N}{2}}
\in L^1(\mathbb{R}^N)$ if $\varepsilon>0$ and $\mu$ satisfies
\eqref{eqn:muN2}. Let $s>0$ and $\eta\in S^{N-1}$. Then the Fourier
transform can be computed by the Lebesgue integral:
\begin{align*}
&\mathcal{F}_{\mathbb{R}^N} (e^{-2\pi\varepsilon r} h_{\mu-\frac{N}{2}})
   (s\eta)
\\
={}&
\int_{S^{N-1}} \int_0^\infty
e^{-2\pi\varepsilon r} r^{\mu+\frac{N}{2}-1}
e^{-2\pi irs\langle\omega,\eta\rangle} dr d\sigma(\omega)
\\
={}&
\int_{S^{N-1}} \mathcal{F}_{\mathbb{R}}
(r_+^{\mu+\frac{N}{2}-1})
(s \langle\omega,\eta\rangle - i\varepsilon) d\sigma(\omega)
\\
={}&
\frac{\Gamma(\mu+\frac{N}{2}) e^{-\frac{\pi i}{2}(\mu+\frac{N}{2})}}
     {(2\pi)^{\mu+\frac{N}{2}}}
\int_{S^{N-1}} (s\langle\omega,\eta\rangle - i\varepsilon)^{-\mu-\frac{N}{2}}
h(\omega) d\sigma(\omega).
\end{align*}
Taking the limit as $\varepsilon\to0$,
we get
\begin{equation*}
\mathcal{F}_{\mathbb{R}^N} h_{\mu-\frac{N}{2}} (s\eta)
=
\frac{\Gamma(\mu+\frac{N}{2}) e^{-\frac{\pi i}{2}(\mu+\frac{N}{2})}}
     {(2\pi)^{\mu+\frac{N}{2}} s^{\mu+\frac{N}{2}}}
\int_{S^{N-1}} (\langle\omega,\eta\rangle - i0)^{-\mu-\frac{N}{2}}
h(\omega) d\sigma(\omega),
\end{equation*}
where $(\langle\omega,\eta\rangle - i0)^\lambda$ denotes the substitution
of $x=\langle\omega,\eta\rangle$ into the distribution
$(x-i0)^\lambda$ (see Example \ref{ex:Riesz}).
Since $h$ is an even function,
the above integral amounts to
\[
e^{i\frac{\pi}{2}(\mu+\frac{N}{2})}
\cos\frac{\pi}{2}\left(\mu+\frac{N}{2}\right)
\int_{S^{N-1}} |\langle\omega,\eta\rangle|^{-\mu-\frac{N}{2}}
h(\omega) d\sigma(\omega)
\]
by \eqref{eqn:xi0}.
Therefore,
$\mathcal{F}_{\mathbb{R}^N} h_{\mu-\frac{N}{2}} (s\eta)$ equals
\begin{align*}
& (2\pi)^{-\mu-\frac{N}{2}} s^{-\mu-\frac{N}{2}}
  \Gamma\left(\mu+\frac{N}{2}\right) \cos\frac{\pi}{2}\left(\mu+\frac{N}{2}\right)
  \int_{S^{N-1}} |\langle\omega,\eta\rangle|^{-\mu-\frac{N}{2}}
  h(\omega)d\sigma(\omega)
\\
&= C_N(\mu)^{-1} s^{-\mu-\frac{N}{2}} (\mathcal{Q}_\mu h) (\eta).
\end{align*}
Thus, Proposition \ref{prop:QF} has been proved.
\end{proof}

So far, $N$ has been an arbitrary positive integer.
Suppose now that $N$ is an even integer, say,
 $N=2n$.
We introduce the symplectic Fourier transform
defined by the formula:
\[
(\mathcal{F}_J f)(Y)
:= \int_{\mathbb{R}^{2n}} f(X) e^{-2\pi i[X,Y]} dX.
\]
We identify $\mathbb{R}^{2n}$ with $\mathbb{C}^n$ by $(x,\xi)
\mapsto x+i\xi$. Correspondingly, the complex structure on
$\mathbb{R}^{2n}$ is given by the
 linear transform
$$
J: \mathbb{R}^{2n} \to \mathbb{R}^{2n},
\quad
J(x,\xi) := (-\xi,x).
$$
Then the formula \eqref{eqn:BJ} is equivalent to
\[
[X,Y] = \langle X,JY \rangle
\quad
(X, Y \in \mathbb{R}^{2n}),
\]
and therefore, our $\mathcal{F}_J$ and the usual Fourier transform
$\mathcal{F}_{\mathbb{R}^{2n}}$ are related by the formula:
\begin{equation}\label{eqn:FJF}
(\mathcal{F}_J f)(Y) = \mathcal{F}_{\mathbb{R}^{2n}}(JY).
\end{equation}

Likewise,
the linear operators $\mathcal{T}_\mu$ (see \eqref{eqn:Tmu}) and
$\mathcal{Q}_\mu$ for $N=2n$ (see \eqref{eqn:Qmu}) are related by
\[
\mathcal{T}_\mu f(Y) = \mathcal{Q}_\mu(JY).
\]
Therefore, Proposition \ref{prop:QF} leads us to:
\begin{proposition}\label{prop:TF}
Let $C_N(\mu)$ be the constant defined in \eqref{eqn:CNmu}.
Then,
\[
\mathcal{T}_\mu = C_{2n}(\mu) \mathcal{F}_J \big|_{V_{-\mu}}.
\]
\end{proposition}

\begin{remark}\label{rem:TF}
Since the symplectic Fourier transform $\mathcal{F}_J$ induces a
bijection $\mathcal{F}_J \big|_{V_{-\mu}} : V_{-\mu}
\stackrel{\sim}{\to} V_{\mu}$ for all $\mu\in\mathbb{C}$,
Proposition \ref{prop:TF} implies that $\mathcal{T}_\mu$ is also
bijective as far as $C_{2n}(\mu) \ne 0, \infty$.

We note that $C_{2n}(\mu)$ has  simple zeros at
$\mu+n = 0,-2,-4,\dotsc$.
In this case, the kernel $\bigl| [X,Y] \bigr|^{-\mu-n}$ is a
polynomial in $Y$ of degree $-(\mu+n)$,
and correspondingly, $(\mathcal{T}_\mu f)(Y)$ is also a polynomial of
the same degree.
Thus, $\operatorname{Image}\mathcal{T}_\mu$ is finite dimensional,
and $\operatorname{Ker}\mathcal{T}_\mu$ is infinite dimensional.

On the other hand, $C_{2n}(\mu)$ has simple poles at $\mu+n =
1,3,5,\dotsc$. This corresponds to the fact that the distribution
$|x|^\lambda$ of one variable has simple poles at
$\lambda=-1,-3,-5,\dotsc$ (see \cite{G}).
\end{remark}

We are now ready to complete the proof of Theorem \ref{thm:Teigen}.
\begin{proof}[Proof of Theorem \ref{thm:Teigen}]
Suppose $p \in \mathcal{H}^{\alpha,\beta}(\mathbb{C}^n)$.
Since $J$ acts on $z_j$ $(1 \le j \le n)$ by $\sqrt{-1}$ and
$\bar{z}_j$ by $-\sqrt{-1}$,
we have
\begin{equation}\label{eqn:pJ}
p(J\eta) = (-1)^{\frac{\alpha-\beta}{2}} p(\eta).
\end{equation}
In view of Lemma \ref{lem:Boch},
Proposition \ref{prop:TF}, and \eqref{eqn:FJF},
the operator $\mathcal{T}_\mu$ acts on
$\mathcal{H}^{\alpha,\beta}(\mathbb{C}^n)$ as a scalar
\[
(-1)^{\frac{\alpha-\beta}{2}} C_{2n}(\mu) B_{2n}(\mu-n,\alpha+\beta).
\]
This amounts to $(-1)^\beta A_{\alpha+\beta}(\mu)$,
whence Theorem \ref{thm:Teigen}.
\end{proof}

\section{Proof of Theorem
\ref{thm:gBZ}} \label{sec:3}

\subsection{Dimension formulas for spherical harmonics}

This subsection summarizes some elementary results on the dimensions
of harmonic polynomials in a way that we shall use later. They are
more or less known, however, we give a brief account of them for the
convenience of the reader.

Let $\mathcal{P}^k(\mathbb{R}^N)$ be the complex vector space of
homogeneous polynomials in $N$ variables of degree $k$. Its
dimension is given by the binomial coefficient:
\[
\dim \mathcal{P}^k(\mathbb{R}^N)
= \begin{pmatrix} k+N-1 \\ k \end{pmatrix}.
\]
In light of the linear bijection
(see e.g.\ \cite[pp.~17]{xhel}):
\[
\mathcal{H}^k(\mathbb{R}^N) \oplus \mathcal{P}^{k-2}(\mathbb{R}^N)
\stackrel{\sim}{\to}
\mathcal{P}^k(\mathbb{R}^N),
\quad
(p,q) \mapsto p(X) + |X|^2 q(X),
\]
we get the dimension formula of $\mathcal{H}^k(\mathbb{R}^N)$:
\begin{align}\label{eqn:HSS}
\dim\mathcal{H}^k(\mathbb{R}^N)
&= \dim \mathcal{P}^k(\mathbb{R}^N) - \dim \mathcal{P}^{k-2}(\mathbb{R}^N)
\nonumber
\\
&= \frac{(k+N-3)! (2k+N-2)}{k! (N-2)!}.
\end{align}
In the next subsection, we shall use the following recurrence formula:
\begin{lemma}\label{lem:Hkrec}
$\displaystyle \dim \mathcal{H}^k(\mathbb{R}^N) + \dim
\mathcal{H}^{k-1}(\mathbb{R}^{N+1}) = \dim
\mathcal{H}^k(\mathbb{R}^{N+1}).
$
\end{lemma}

\begin{proof}
By the elementary combinatorial formula
\[
\begin{pmatrix} m \\ k \end{pmatrix}
+ \begin{pmatrix} m \\ k-1 \end{pmatrix}
= \begin{pmatrix} m+1 \\ k \end{pmatrix},
\]
we have
\begin{equation}\label{eqn:Skk}
\dim \mathcal{P}^k(\mathbb{R}^N) + \dim \mathcal{P}^{k-1}(\mathbb{R}^{N+1})
= \dim \mathcal{P}^k(\mathbb{R}^{N+1}).
\end{equation}
Taking the difference of \eqref{eqn:Skk} for $k$ and $k-2$,
and applying \eqref{eqn:HSS},
we get Lemma \ref{lem:Hkrec}.
\end{proof}

To find the dimension formula of
$\mathcal{H}^{\alpha,\beta}(\mathbb{C}^n)$ one might apply the above
method (see e.g.\ \cite[Section 11.2.1]{VK}), but it would be more
convenient for our purpose to use representation theory. There is a
natural action of the unitary group $U(n)$ on
$\mathcal{H}^{\alpha,\beta}(\mathbb{C}^n)$. This representation is
irreducible, and its highest weight is given by $\displaystyle
(\alpha,0,\dots,0,-\beta) $ in the standard coordinates of the
Cartan subalgebra. By the Weyl character formula, we get
\[
\dim \mathcal{H}^{\alpha,\beta}(\mathbb{C}^n)
= \frac{(\alpha+\beta+n-1)}
       {(n-1)! (n-2)!}\prod_{i=2}^{n-1} (\alpha+i-1)(\beta+n-i)
.
\]
If we use the Pochhammer symbol $(a)_l$ defined by
\[
(a)_l
:= \frac{\Gamma(a+l)}{\Gamma(a)}
=  a(a+1) \cdots (a+l-1),
\]
then we may express these dimensions as
\begin{alignat}{2}\label{eqn:Hab}
&\dim \mathcal{H}^k(\mathbb{R}^N)
&&= \frac{(k+1)_{N-3} (2k+N-2)}{\Gamma(N-1)},
\nonumber
\\
&\dim \mathcal{H}^{\alpha,\beta}(\mathbb{C}^n)
&&= \frac{(\alpha+\beta+n-1) (\alpha+1)_{n-2} (\beta+1)_{n-2}}
       {\Gamma(n)\Gamma(n-1)}.
\end{alignat}

\subsection{Alternating sum of $\dim\mathcal{H}^{\alpha,\beta}(\mathbb{C}^n)$}

By the direct sum decomposition \eqref{eqn:Habk},
the following identity is obvious:
\[
\dim\mathcal{H}^k(\mathbb{R}^{2n})
= \sum_{\alpha+\beta=k}
  \dim\mathcal{H}^{\alpha,\beta}(\mathbb{C}^n).
\]
However, what we need for the proof of Theorem \ref{thm:gBZ}
is an explicit formula for the alternating sum:
\[
D(k)
:= \sum_{\alpha+\beta=k} (-1)^\beta
\dim \mathcal{H}^{\alpha,\beta} (\mathbb{C}^n).
\]
Clearly, $D(k)=0$ for odd $k$ because
$\dim \mathcal{H}^{\alpha,\beta} (\mathbb{C}^n)
 = \dim \mathcal{H}^{\beta,\alpha} (\mathbb{C}^n)$.

A closed formula of $D(k)$ for even $k$ is the main issue of this
subsection, and we establish the following relation:
\begin{proposition}\label{lem:Dk}
\begin{equation}\label{eqn:Dk}
D(2l)
= \dim \mathcal{H}^l(\mathbb{R}^{n+1})
= \frac{(n-1)_l (\frac{n+1}{2})_l}
       {l! (\frac{n-1}{2})_l} .
\end{equation}
\end{proposition}

\begin{remark}\label{rem:Dk}
The Pochhammer symbol
$(a)_l$ may be regarded as a meromorphic function.
Thus, the right-hand side of \eqref{eqn:Dk} can be regarded as a
meromorphic function of $n$.
In this sense,
the right-hand side of
\eqref{eqn:Dk} still makes sense for $n=1$.
\end{remark}

The rest of this subsection is devoted to the proof of Proposition
\ref{lem:Dk}.
For this,
we set
\begin{equation*}
X^{(l)} := x^l + \frac{1}{x^l}
\quad\text{for $l=1,2,\dotsc$}.
\end{equation*}
It is readily seen that $X^{(l)}$ is expressed as a monomial in
$$
X := x+\frac{1}{x}
$$
of degree $l$. For example,
\begin{equation}\label{eqn:XN2}
X^{(1)} = X, \
X^{(2)} = X^2-2, \
X^{(3)} = X^3-3X,
\dotsc.
\end{equation}
For an arbitrary $l$, we have the following formula:
\begin{lemma}\label{lem:3.3}
\begin{equation}\label{eqn:XN}
X^{(l)}
= \sum_{j=0}^{[\frac{l}{2}]} (-1)^j \dim
   \mathcal{H}^j (\mathbb{R}^{l+2-2j}) X^{l-2j}.
\end{equation}
\end{lemma}


\begin{proof}
We prove Lemma \ref{lem:3.3} by induction on $l$. The equation
\eqref{eqn:XN} holds for $l=1,2$ by \eqref{eqn:XN2}. Suppose $l \ge
2$. We shall prove the equation \eqref{eqn:XN} for $l+1$. We use
\begin{align*}
X^{(l+1)} &= \left(x+\frac{1}{x}\right) \left(x^l +
\frac{1}{x^l}\right) - \left(x^{l-1} + \frac{1}{x^{l-1}}\right)
\\
&= X X^{(l)} - X^{(l-1)}.
\end{align*}
By substituting \eqref{eqn:XN} for $l$ and $l-1$ into the right-hand
side, we get
\begin{align*}
X^{(l+1)}
={}& \sum_{j=0}^{[\frac{l}{2}]} (-1)^j \dim
   \mathcal{H}^j(\mathbb{R}^{l+2-2j}) X^{l+1-2j}
\\
&  -  \sum_{i=0}^{[\frac{l-1}{2}]} (-1)^i \dim
    \mathcal{H}^i(\mathbb{R}^{l+1-2i}) X^{l-1-2i}
\\
={}& X^{l+1} + \sum_{j=1}^{[\frac{l+1}{2}]}\left( (-1)^j
   (\dim \mathcal{H}^j(\mathbb{R}^{l+2-2j})
     + \dim \mathcal{H}^{j-1}(\mathbb{R}^{l+3-2j})) X^{l+1-2j}\right).
\end{align*}
To see the second equality for odd $l$, we note that $\dim
\mathcal{H}^d(\mathbb{R}^1) = 0$ for $d \ge 2$, and thus
\begin{equation}\label{N:odd}
\dim \mathcal{H}^j(\mathbb{R}^{l+2-2j}) = 0
\quad \text{for $j = \frac{l+1}{2}$}.
\end{equation}

Applying the recurrence formula given in Lemma \ref{lem:Hkrec}, we
get \eqref{eqn:XN} for $l+1$. By induction, we have proved Lemma
\ref{lem:3.3}.
\end{proof}

\begin{proof}[Proof of Proposition \ref{lem:Dk}]
We take a maximal torus $T$ of $U(n)$ and its coordinate
$(x_1,\dots,x_n)$ such that
\begin{equation*}
T = \{x = (x_1,\dots,x_n) \in \mathbb{C}^n :
       |x_1| = \dots = |x_n| = 1 \},
\end{equation*}
and that the linear map
$J : \mathbb{R}^{2n} \to \mathbb{R}^{2n}$
is represented as
$J = (\sqrt{-1},\dots,\sqrt{-1}) \in T$.
Then the character $\chi_{\mathcal{H}^k(\mathbb{R}^{2n})} (g)$ of
the representation of $O(2n)$ on
$\mathcal{H}^k(\mathbb{R}^{2n})$ takes the value
\[
\sum_{\alpha+\beta=k} (-1)^{\frac{\alpha-\beta}{2}}
\dim \mathcal{H}^{\alpha,\beta}(\mathbb{C}^n)
= (-1)^{\frac{k}{2}} D(k)
\]
at $g = J$.

By using this observation,
we shall analyze the character $\chi_{\mathcal{H}^k(\mathbb{R}^{2n})}(g)$
as $g$ approaches to
 the singular point $J\in T$.

Let
\[
X_j^{(l)}
:= x_j^l + \frac{1}{x_j^l}
\quad
(1 \le j\le n, \  l \in \mathbb{N}),
\]
and we set
\[
s_k(x)
:= \det
\left(
   \begin{array}{llcl}
      X_1^{(k+n-1)} & X_2^{(k+n-1)} & \cdots & X_n^{(k+n-1)}
   \\
      X_1^{(n-2)}   & X_2^{(n-2)}   & \cdots & X_n^{(n-2)}
   \\
      \ \vdots      & \ \vdots      &        & \ \vdots
   \\
      X_1^{(1)}     & X_2^{(1)}     &        & X_n^{(1)}
   \\
      \ 1           & \ 1           & \cdots & \ 1
   \end{array}
\right).
\]
Then, by the Weyl character formula for the group $O(2n)$ and by
using a trick which reduces the summation over the Weyl group for
$O(2n)$ to that over the symmetric group $\mathcal{S}_n$ (see
\cite{P}), we have
\[
\chi_{\mathcal{H}^k(\mathbb{R}^{2n})} (x)
= \frac{s_k(x)}{s_0(x)}
\quad\text{for $x \in T$}.
\]

Since $X^{(l)} \equiv X^l \bmod \text{$\mathbb{Q}$-span} \langle
1,X,\dots,X^{l-1} \rangle$ an elementary property of the determinant
shows:
\[
s_k(x)
= \det
\left(
     \begin{array}{llcl}
      X_1^{(k+n-1)} & X_2^{(k+n-1)} & \cdots & X_n^{(k+n-1)}
   \\
      X_1^{n-2}     & X_2^{n-2}     & \cdots & X_n^{n-2}
   \\
      \ \vdots      & \ \vdots      &        & \ \vdots
   \\
      X_1           & X_2           &        & X_n
   \\
      \ 1           & \ 1           & \cdots & \ 1
   \end{array}
\right).
\]
As $x_j$ goes to $\sqrt{-1}$,
$X_j$ tends to $0$ $(1 \le j \le n)$.
Therefore, we have
\begin{align*}
\chi_{\mathcal{H}(\mathbb{R}^{2n})}^{2l} (J)
&= \lim_{X_1,\dots,X_N\to 0} \frac{s_{2l}(x)}{s_0(x)}
\\
&= \parbox[t]{20em}{the coefficient of $X^{n-1}$ in the expansion \eqref{eqn:XN}
         for $X^{(2l+n-1)}$}
\\
&= (-1)^l \dim \mathcal{H}^l(\mathbb{R}^{n+1}).
\end{align*}
Here, we have used Lemma \ref{lem:3.3} for the last equality.
Thus, we have proved
\[
D(2l) = \dim \mathcal{H}^l(\mathbb{R}^{n+1}).
\]
The second equality of \eqref{eqn:Dk} is immediate from \eqref{eqn:HSS}.
\end{proof}

\subsection{Triple integral as a Trace}\label{sec:33}
We are now ready to prove Theorem \ref{thm:gBZ}.
As we remarked in Introduction,
the both sides of Theorem \ref{thm:gBZ} are
meromorphic functions of $\lambda_1$, $\lambda_2$, and $\lambda_3$.
Therefore,
it is sufficient to prove the identity in Theorem \ref{thm:gBZ} in an
open set of the parameters $(\lambda_1,\lambda_2,\lambda_3) \in \mathbb{C}^3$.

By the change of variables $\mu_j := \frac{1}{2}
(\lambda_1+\lambda_2+\lambda_3-n) - \lambda_j$ $(1 \le j \le 3)$, we
first consider the case when $\operatorname{Re}\mu_1 \ll 0$,
$\operatorname{Re}\mu_2 \ll 0$, and $\operatorname{Re}\mu_3 \ll 0$.
Then, the operators $\mathcal{T}_{\mu_1}$, $\mathcal{T}_{\mu_2}$,
and $\mathcal{T}_{\mu_3}$ are Hilbert--Schmidt operators on
$L^2(S^{2n-1})$. In particular, the composition $\mathcal{T}_{\mu_1}
\mathcal{T}_{\mu_2} \mathcal{T}_{\mu_3}$ is of trace class, and its
trace is given by
\begin{align*}
&\operatorname{Trace}(\mathcal{T}_{\mu_1} \mathcal{T}_{\mu_2} \mathcal{T}_{\mu_3})
\\
&= \int_{(S^{2n-1})^3
}
\bigl|[X,Y]\bigr|^{-\mu_1-n}
\bigl|[Y,Z]\bigr|^{-\mu_2-n}
\bigl|[Z,X]\bigr|^{-\mu_3-n}
d\sigma(X) d\sigma(Y) d\sigma(Z).
\end{align*}
On the other hand,
the trace of the operator
$\mathcal{T}_{\mu_1} \mathcal{T}_{\mu_2} \mathcal{T}_{\mu_3}$
can be also computed by its eigenvalues.
Therefore, by using Theorem \ref{thm:Teigen},
we have
\begin{align*}
\operatorname{Trace}(\mathcal{T}_{\mu_1} \mathcal{T}_{\mu_2} \mathcal{T}_{\mu_3})
&= \sum_{\alpha,\beta}
   \left( \prod_{j=1}^3 (-1)^\beta A_{\alpha+\beta} (\mu_j) \right)
   \dim \mathcal{H}^{\alpha,\beta} (\mathbb{C}^n)
\\
&= \sum_{k=0}^\infty \prod_{j=1}^3 A_k(\mu_j)
   \left( \sum_{\alpha+\beta=k} (-1)^{3\beta}
          \dim \mathcal{H}^{\alpha,\beta} (\mathbb{C}^n) \right)
\\
&= \sum_{l=0}^\infty D(2l) \prod_{j=1}^3 A_{2l} (\mu_j).
\end{align*}

Applying Proposition \ref{lem:Dk},
we get
\begin{equation*}
\operatorname{Trace} (\mathcal{T}_{\mu_1} \mathcal{T}_{\mu_2} \mathcal{T}_{\mu_3})
= \sum_{l=0}^\infty
   A_{2l}(\mu_1) A_{2l}(\mu_2) A_{2l}(\mu_3)
   \dim \mathcal{H}^l(\mathbb{R}^{n+1}).
\end{equation*}

In light of the recurrence relation:
\[
\frac{A_{2l+2}(\mu)}{A_{2l}(\mu)}
= \frac{l+\frac{n+\mu}{2}}{l+\frac{n-\mu}{2}},
\]
the meromorphic function $A_{2l}(\mu)$ can be expressed in terms of
Pochhammer symbols as
\[
A_{2l}(\mu)
= \frac{(\frac{n+\mu}{2})_l}{(\frac{n-\mu}{2})_l} A_0(\mu),
\]
where
\begin{equation}\label{eqn:A0}
A_0(\mu)
= 2\pi^{n-\frac{1}{2}}
  \frac{\Gamma(\frac{1-n-\mu}{2})}{\Gamma(\frac{n-\mu}{2})}.
\end{equation}

Therefore,
\begin{align*}
&\operatorname{Trace}(\mathcal{T}_{\mu_1} \mathcal{T}_{\mu_2} \mathcal{T}_{\mu_3})
\\
&= A_0(\mu_1) A_0(\mu_2) A_0(\mu_3)
   \sum_{l=0}^\infty \frac{(n-1)_l (\frac{n+1}{2})_l}
                          {l! (\frac{n-1}{2})_l}
   \prod_{j=1}^3 \frac{(\frac{n+\mu_j}{2})_l}
                      {(\frac{n-\mu_j}{2})_l}
\\
&= A_0(\mu_1) A_0(\mu_2) A_0(\mu_3)
   \, {}_5F_4
   \left( \begin{matrix}
             n-1 & \frac{n+1}{2} & \frac{n+\mu_1}{2}
                 & \frac{n+\mu_2}{2} & \frac{n+\mu_3}{2}
          \\[1ex]
                 & \frac{n-1}{2} & \frac{n-\mu_1}{2}
                 & \frac{n-\mu_2}{2} & \frac{n-\mu_3}{2}
          \end{matrix}
          \   ; 1
   \right).
\end{align*}
Here ${}_5F_4$ is a generalized hypergeometric function.

A generalized hypergeometric function
\[
{}_pF_q
\left( \begin{matrix}
         \alpha_1 & \alpha_2 & \cdots & \alpha_p
       \\
                  & \beta_1  & \cdots & \beta_q
       \end{matrix}
      \  ; z
\right)
\]
is called \textit{well-poised} (see \cite{B}) if $p=q+1$ and
\[
1+\alpha_1 = \alpha_2+\beta_1 = \dots = \alpha_p+\beta_q.
\]
In particular, our case is well-poised,
and we can use the following
\textit{Dougall--Ramanujan identity}
(see [loc.~cit., pp.~25--26]):
\begin{align*}
&{}_5F_4
\left( \begin{matrix}
          m-1 & \frac{m+1}{2} & -x & -y & -z
       \\[1ex]
              & \frac{m-1}{2} & x+m & y+m &z+m
       \end{matrix}
       \  ; 1
\right)
\\
&= \frac{\Gamma(x+m) \Gamma(y+m) \Gamma(z+m) \Gamma(x+y+z+m)}
        {\Gamma(m) \Gamma(x+y+m) \Gamma(y+z+m) \Gamma(x+z+m)}.
\end{align*}
Together with \eqref{eqn:A0},
we get
\begin{equation}\label{eqn:TrT}
\operatorname{Trace} (\mathcal{T}_{\mu_1} \mathcal{T}_{\mu_2} \mathcal{T}_{\mu_3})
= \frac{(2\pi^{n-\frac{1}{2}})^3 \Gamma(\frac{1-n-\mu_1}{2})
          \Gamma(\frac{1-n-\mu_2}{2}) \Gamma(\frac{1-n-\mu_3}{2})
          \Gamma(\frac{-\mu_1-\mu_2-\mu_3-n}{2})}
        {\Gamma(n) \Gamma(-\frac{\mu_1+\mu_2}{2})
          \Gamma(-\frac{\mu_2+\mu_3}{2})
          \Gamma(-\frac{\mu_1+\mu_3}{2})}.
\end{equation}
Now, Theorem \ref{thm:gBZ} follows by substituting $\mu_1 =
-\frac{1}{2}(\alpha+n)$, $\mu_2 = -\frac{1}{2}(\beta+n)$, and $\mu_3
= -\frac{1}{2}(\gamma+n)$.

\section{Other triple integral formulas}
\label{sec:4}

In this section, we discuss explicit formulas for the integrals of
the triple product of powers of $|x-y|$ and $|\langle x,y\rangle|$
instead of those of the symplectic form $\bigl|[X,Y]\bigr|$.

\subsection{Triple product of powers of $|x-y|$}

In this subsection we consider a  family of linear operators that
depend meromorphically on $\mu\in\mathbb C$ by
\[
\mathcal{R}_\mu :
C^\infty(S^m) \to C^\infty(S^m)
\]
defined by
 \begin{equation}\label{eqn:Rmu}
 (\mathcal{R}_\mu f)(\eta) =
\int_{S^m} |\omega-\eta|^{-\mu-m} f(\omega) d\sigma(\omega).
\end{equation}
The multiplier action of $\mathcal{R}_\mu$ on spherical harmonics is
known (see e.g. \cite{Bc}):
\begin{equation}\label{eqn:Reigen}
\mathcal{R}_\mu \Big|_{\mathcal{H}^k(\mathbb{R}^{m+1})}
= \gamma_k(\mu) \operatorname{id},
\end{equation}
where
$\gamma_k(\mu) \equiv \gamma_k(\mu,\mathbb{R}^{m+1})$
is given by
\begin{equation}\label{eqn:gamma}
\gamma_k(\mu)
= \frac{\Gamma(m+\frac{1}{2}) \Gamma(-\frac{\mu}{2})
        \Gamma(k+\frac{m+\mu}{2})}
       {2^{\mu+1} \sqrt{\pi} \Gamma(\frac{\mu+m}{2})
        \Gamma(k+\frac{m-\mu}{2})}
.
\end{equation}

Then, by an argument parallel to Section \ref{sec:33}, we can obtain
a closed formula for the triple integral built on $\mathcal R_\mu$
(see Theorem \ref{thm:SO} below). Instead of repeating similar
computations, we pin down a comparison result between the two triple
integral formulas by using Proposition \ref{lem:Dk}. This comparison
result explains the reason why the same method
(e.g.~Dougall--Rama\-nujan identity) is applicable, and seems
interesting for its own sake.

\begin{proposition}\label{lem:SpSO}
\begin{align}
&\operatorname{Trace}\label{lem:41}
 (\mathcal{R}_{\mu_1}\mathcal{R}_{\mu_2}\mathcal{R}_{\mu_3} :
 L^2(S^m) \to L^2(S^m))\nonumber
\\
&= c \operatorname{Trace}
   (\mathcal{T}_{\mu_1}\mathcal{T}_{\mu_2}\mathcal{T}_{\mu_3} :
   L^2(S^{2m-1}) \to L^2(S^{2m-1})),
\end{align}
where
\[
c = \left( \frac{\Gamma(m+\frac{1}{2})}{2^2\pi^m} \right)^3
    \prod_{j=1}^3
    \frac{\Gamma(-\frac{\mu_j}{2})}
         {2^{\mu_j}\Gamma(\frac{-\mu_j-m+1}{2})}.
\]
\end{proposition}

\begin{proof}
By \eqref{eqn:Reigen} the left-hand side of (\ref{lem:41}) equals
\begin{equation}\label{eqn:mur}
 \sum_{k=0}^\infty
   \left( \prod_{j=1}^3 \gamma_k(\mu_j,\mathbb{R}^{m+1}) \right)
   \dim \mathcal{H}^k(\mathbb{R}^{m+1}).
\end{equation}

Comparing (\ref{eqn:gamma}) with Theorem \ref{thm:Teigen} we get
\begin{equation}\label{eqn:RCN}
\frac{\gamma_k(\mu,\mathbb{R}^{m+1})}{A_{2k}(\mu,\mathbb{C}^m)}
= \frac{\Gamma(m+\frac{1}{2}) \Gamma(-\frac{\mu}{2})}
       {2^{\mu+2} \pi^m \Gamma(\frac{-\mu-m+1}{2})}.
\end{equation}

By \eqref{eqn:RCN} and \eqref{eqn:TrT}, we see that (\ref{eqn:mur})
equals the right-hand side of (\ref{lem:41}).
\end{proof}

The right-hand side in Proposition \ref{lem:SpSO} was found in
\eqref{eqn:TrT}. Then, by a simple computation, we get
\begin{align*}
&\operatorname{Trace} (\mathcal{R}_{\mu_1}\mathcal{R}_{\mu_2}\mathcal{R}_{\mu_3})
\\
&= \left( \frac{\Gamma(m+\frac{1}{2})}
                 {2 \pi^{\frac{1}{2}} }
     \right)^3
   \frac{\Gamma(\frac{-\mu_1-\mu_2-\mu_3-m}{2})}
        {\Gamma(m)}
  \prod_{j=1}^3
  \frac{\Gamma(-\frac{\mu_j}{2})}
       {2^{\mu_j}\Gamma(\frac{\mu_j-(\mu_1+\mu_2+\mu_3)}{2})}.
\end{align*}
Finally, substituting $\mu_j =
\frac{1}{2}(\lambda_1+\lambda_2+\lambda_3-m)-\lambda_j$ $(1 \le j
\le 3)$, we have proved the following:
\begin{theorem}\label{thm:SO}
Let $\alpha,\beta,\gamma$, and $\delta$ be as in Theorem \ref{thm:gBZ}
\begin{align*}
&\int_{S^m\times S^m\times S^m}
 |Y-Z|^{\frac{\alpha-m}{2}}
 |Z-X|^{\frac{\beta-m}{2}}
 |X-Y|^{\frac{\gamma-m}{2}}
 d\sigma(X) d\sigma(Y) d\sigma(Z)
\\
&=
    \left( \frac{\Gamma(m+\frac{1}{2})}
                {2^{1-\frac{m}{2}} \pi^{\frac{1}{2}}}
    \right)^3
    \frac{1}
         {2^{\frac{\lambda_1+\lambda_2+\lambda_3}{2}}\Gamma(m)}
    \frac{\Gamma(\frac{\alpha+m}{4})
          \Gamma(\frac{\beta+m}{4})
          \Gamma(\frac{\gamma+m}{4})
          \Gamma(\frac{\delta+m}{4})}
         {\Gamma(\frac{m-\lambda_1}{2})
          \Gamma(\frac{m-\lambda_2}{2})
          \Gamma(\frac{m-\lambda_3}{2})}
.
\end{align*}
\end{theorem}
We will give in Proposition \ref{prop:6.9}
 the domain for the absolute convergence of the above integral.

\begin{remark}
The formula in Theorem \ref{thm:SO} was previously found by A.
Deitmar \cite{D} by a different method; namely it established a
recurrence formula bridging $SO_o(\ell+1,1)$ to $SO_o(\ell-1,1)$ and
used the Bernstein--Reznikov formula for $SO_o(2,1)$ and an
analogous formula for $SO_o(3,1)$.
\end{remark}

\subsection{Triple product of powers of $|\langle x,y\rangle|$}

In this subsection we consider the third case, namely, the linear
operators $\mathcal{Q}_\mu : C^\infty(S^{N-1}) \to
C^\infty(S^{N-1})$ defined by the kernel $|\langle
x,y\rangle|^{-\mu-\frac{N}{2}}$ (see \eqref{eqn:Qmu}) and the
corresponding triple product integrals.

Here is the counterpart of Theorem \ref{thm:Teigen} for
$\mathcal{Q}_\mu$:
\begin{proposition}\label{prop:Qeigen}
$\mathcal{Q}_\mu \Big|_{\mathcal{H}^k(\mathbb{R}^N)} = 0$
for odd $k$, and
\[
\mathcal{Q}_\mu \Big|_{\mathcal{H}^{2l}(\mathbb{R}^N)}
= c_N(\mu,l) \operatorname{id},
\]
where
\[
c_N(\mu,l)
= (-1)^l
  \frac{2\pi^{\frac{N-1}{2}} \Gamma(\frac{2-N-2\mu}{4})
         \Gamma(l+\frac{2\mu+N}{4})}
       {\Gamma(\frac{N+2\mu}{4}) \Gamma(l+\frac{-2\mu+N}{4})}.
\]
\end{proposition}

\begin{proof}
By Lemma \ref{lem:Boch} and Proposition \ref{prop:QF},
we have
\[
c_N(\mu,l)
= C_N(\mu) B_N(\mu-\frac{N}{2},2l).
\]
\end{proof}
As in the previous cases,
 we have
\begin{align}\label{eqn:TrQ1}
&\operatorname{Trace}
(\mathcal{Q}_{\mu_1}\mathcal{Q}_{\mu_2}\mathcal{Q}_{\mu_3} :
 L^2(S^{N-1}) \to  L^2(S^{N-1}))
\nonumber
\\
&
=
\sum_{l=0}^\infty
\left( \prod_{j=1}^3 c_N(\mu_j,l) \right)
\dim \mathcal{H}^{2l}(\mathbb{R}^N).
\end{align}
By substituting
\begin{align*}
c_N(\mu,l)
&= (-1)^l c_N(\mu,0)
  \frac{(\frac{N+2\mu}{4})_l}
       {(\frac{N-2\mu}{4})_l},
\\
\dim \mathcal{H}^{2l}(\mathbb{R}^N)
&= \frac{(\frac{N}{2}-1)_l (\frac{N-1}{2})_l (\frac{N+2}{4})_l}
        {l! (\frac{1}{2})_l (\frac{N-2}{4})_l}
\end{align*}
into the right-hand side of \eqref{eqn:TrQ1},
we see that \eqref{eqn:TrQ1} equals
\begin{align*}
&\left( \prod_{j=0}^3 c_N(\mu_j,0) \right)
\sum_{j=0}^\infty (-1)^l
\prod_{j=1}^3 \frac{(\frac{N+2\mu_j}{4})_l}
                  {(\frac{N-2\mu_j}{4})_l}
   \, \frac{(\frac{N}{2}-1)_l (\frac{N-1}{2})_l (\frac{N+2}{4})_l}
           {l! (\frac{1}{2})_l (\frac{N-2}{4})_l}
\\
&= \prod_{j=0}^3 c_N(\mu_j,0)
   {}_6F_5
   \left( \begin{matrix}
             \frac{N}{2}-1 & \frac{N+2}{4} & \frac{N-1}{2}
                           & \frac{N+2\mu_1}{4} & \frac{N+2\mu_2}{4}
                           & \frac{N+2\mu_3}{4}
          \\[1ex]
                           & \frac{N-2}{4} & \frac{1}{2}
                           & \frac{N-2\mu_1}{4} & \frac{N-2\mu_2}{4}
                           & \frac{N-2\mu_3}{4}
          \end{matrix}
          \   ; 1
   \right).
\end{align*}
By using Whipple's transformation (\cite[p.28]{B}):
\begin{align*}
&{}_6F_5
\left(
\begin{matrix}
    a, & 1+\frac{1}{2}a, & b, & c, & d, & e
 \\
   & \frac{1}{2}a, & 1+a-b, & 1+a-c, & 1+a-d, &1+a-e
\end{matrix}
\  ; -1
\right)
\\
&= \frac{\Gamma(1+a-d)\Gamma(1+a-e)}
        {\Gamma(1+a)\Gamma(1+a-d-e)}
\\
&\quad \times {}_3F_2
\left(
\begin{matrix}
    1+a-b-c, & d, & e
 \\
    &1+a-b, &1+a-c
\end{matrix}
\  ; 1
\right),
\end{align*}
we get
\begin{align}\label{eqn:F32}
\operatorname{Trace}
(\mathcal{Q}_{\mu_1} \mathcal{Q}_{\mu_2} \mathcal{Q}_{\mu_3})
= {}&
\frac{(2\pi^{\frac{N-3}{2}})^3 \prod_{j=1}^3 \Gamma(\frac{2-N-2\mu_j}{4})}
     {\Gamma (\frac{N}{2}) \Gamma (-\frac{\mu_2+\mu_3}{2})
       \Gamma (\frac{N-2\mu_1}{4})}
\nonumber
\\
&\times
  {}_3F_2
  \left( \begin{matrix}
             \frac{2-N-2\mu_1}{4}  & \frac{N+2\mu_2}{4}
                    & \frac{N+2\mu_3}{4}
         \\[1ex]
                & \frac{1}{2}  & \frac{N-2\mu_1}{4}
         \end{matrix}
         \  ; 1
   \right).
\end{align}
Hence we have proved:
\begin{theorem}\label{thm:xy}
We have the following identity as a meromorphic function of
$(\nu_1,\nu_2,\nu_3)$:
\begin{align*}
& \int_{S^{N-1}\times S^{N-1}\times S^{N-1}}
  |\langle y,z \rangle|^{-2\nu_1} |\langle z,x\rangle|^{-2\nu_2} |\langle x,y\rangle|^{-2\nu_3}
  d\sigma(x) d\sigma(y) d\sigma(z)
\\
&=
\frac{(2\pi^{\frac{N-3}{2}})^3 \prod_{j=1}^3
         \Gamma(\frac{1}{2}-\nu_j)}
     {\Gamma(\frac{N}{2}) \Gamma(-\nu_2-\nu_3+\frac{N}{2})
         \Gamma(-\nu_1+\frac{N}{2})}
\times
{}_3F_2
  \left( \begin{matrix}
             \frac{1}{2}-\nu_1  & \nu_2  & \nu_3
         \\[1ex]
                & \frac{1}{2}  & -\nu_1+\frac{N}{2}
         \end{matrix}
         \  ; 1
   \right).
\end{align*}
\end{theorem}
We will give in Proposition \ref{ex:inner} the precise region for the
absolute convergence of the above integral.

\section{Perspectives from representation theory}\label{sec:5}


In this paper we have proved closed formulas for the triple
integrals (see e.g. Theorem \ref{thm:gBZ}), based on a combination
of methods from classical harmonic analysis.
As we have seen, these methods allow us
to establish explicit formulas for symplectic groups of any rank,
and even in rank one case it gives a new proof of the original
results due to Bernstein and Reznikov \cite{B2} and Deitmar
\cite{D}.

So far we have avoided infinite dimensional representation theory,
which was not used in our proof of main results.
On the other hand,
there are a number of interesting perspectives of these
formulas, and also of the steps in its proof, that deserve comments.

One aspect of
Theorem \ref{thm:gBZ} is that the triple integral considered therein
arises from a particular series of representations $\pi_\mu$ of the
symplectic group $G=Sp(n,\mathbb R)$ of rank $n$ induced from a
maximal parabolic subgroup $P \subset G$ and depending on a complex
parameter $\mu$. Section \ref{sec:5} highlights this point mostly.

Another aspect is that of analytic number theory, which was the main
theme of \cite{B1,B2}. Motivated by the classical Rankin--Selberg
method, authors considered a cocompact discrete subgroup of the rank
one symplectic group and automorphic functions on the associated
locally symmetric space. The product of two such functions may be
decomposed in terms of a basis of automorphic functions and the
corresponding coefficients are related to automorphic $L$-functions.
The closed formula ($n=1$ in Theorem \ref{thm:gBZ}) gave an estimate
of their decay \cite{B2}.

Yet another aspect of
the above mentioned triple integral is that it arises also in
pseudo-differential analysis of the phase space $\mathbb{R}^{2n}$.
This phenomenon was treated in \cite{KOPU}, where the symmetries of
the Weyl operator calculus on the Hilbert space
$L^2(\mathbb{R}^{2n})$ were considered.

\subsection{Invariant trilinear forms}\label{sec:5.2}
Now we focus on some links between the triple integrals discussed in
Sections 1--4 and representation theory of semisimple Lie groups.

We begin with a construction of an invariant trilinear form based on
the Knapp--Stein intertwining operators. Let $G$ be a connected real
semisimple Lie group and $P$ an arbitrary parabolic subgroup. Let
 $P=MAN$ be a Langlands decomposition, $\mathfrak a$ and $\mathfrak n$ the Lie
algebras of $A$ and $N$ respectively, and $2\rho$ the sum of roots
of $\mathfrak n$ with respect to $\mathfrak a$. Take a Cartan
involution $\theta$ of $G$ stabilizing $MA$ and set $K=\{g\in
G\,:\,\theta(g)=g\}.$

For $\lambda\in\mathfrak{a}^*_{\mathbb C}$ we define (possibly
degenerate) principal series representations of $G$, to be denoted
by $\pi_\lambda$, on the space of smooth sections for the
$G$-equivariant line bundle $\mathcal
L_{\lambda+\rho}=G\times_{P}\mathbb C_{\lambda+\rho}$ over the real
flag variety $G/P$, equivalently on the vector space
$$
V_\lambda^\infty\equiv V_\lambda:=\{f\in C^\infty(G)\,:\,
f(gman)=a^{-\lambda-\rho}f(g),\, \forall man\in P\}.
$$
In our parametrization,
$\mathcal{L}_{2\rho}$ is the volume bundle
$\Lambda^{\dim G/P} T^*(G/P)$ over $G/P$.
Similarly, the space of distribution sections for $\mathcal
L_{\lambda+\rho}$ will be denoted by $V_\lambda^{-\infty}$. These
representations are called \emph{spherical} because $V_\lambda$
contains a $K$-fixed vector $\indic_\lambda$ which is defined by the
formula: $\indic_\lambda(kman):=a^{-\lambda-\rho}$ for $kman\in KP$.

Denote by $\overline P=MA\overline N$ the opposite parabolic
subgroup to $P$. Assume that it satisfies the condition:
\begin{description}
\item C1. $P$ and $\overline{P}$ are
conjugate in $G$.
\end{description}
 Then there exists the $G$-intertwining
operators $\displaystyle \mathcal T_\lambda:\, V_{-\lambda}\to
V_{\lambda}$, referred to as the {\it Knapp--Stein intertwining
operators} \cite{KS}, that depend meromorphically on $\lambda$. They
are given by the distribution-valued kernels $K_\lambda(x,y)\in
V_\lambda^{-\infty}\otimes V_\lambda^{-\infty}$ such that
$\displaystyle (\mathcal T_\lambda f)(x)=\langle
f(y),\,K_\lambda(x,y)\rangle\in V_\lambda$ for $f\in V_{-\lambda}$.
The Knapp--Stein kernel $K_{\lambda}$ may be
thought of as a distribution
 on $G \times G$
 subject to the following invariance condition
 for $g \in G$ and $m_j a_j n_j \in MAN$ ($j=1,2$):
\begin{equation}\label{eqn:invKS}
     K_{\lambda}(g x m_1 a_1 n_1, g y m_2 a_2 n_2)
    =a_1^{-\lambda -\rho}a_2^{-\lambda -\rho} K_{\lambda}(x,y).
\end{equation}

For $f_j\in V_{\lambda_j}\,(j=1,2,3),$ we set
\begin{align}\label{eqn:T3l}
&\mathbf{T}_{\lambda_1,\lambda_2,\,\lambda_3}(f_{1},f_2,f_3)
\nonumber
\\
&:=\langle
K_{\frac12(\alpha-\rho)}(y,z)
K_{\frac12(\beta-\rho)}(z,x)K_{\frac12(\gamma-\rho)}(x,y),f_1(x)f_2(y)f_3(z)\rangle,
\end{align}
where
$\alpha=\lambda_1-\lambda_2-\lambda_3,\,\beta=-\lambda_1+\lambda_2-\lambda_3,\,\gamma=-\lambda_1-\lambda_2+\lambda_3\in\mathfrak{a}_\mathbb{C}^*$.

We have the following:
\begin{proposition}\label{prop:51}
Assume $P$ and $\overline P$ are conjugate in $G$.  Then there
exists a non--empty  open region of
$(\lambda_1,\lambda_2,\lambda_3)\in(\mathfrak a_\mathbb C^*)^3$ for
which the integral (\ref{eqn:T3l}) converges. It extends as a
meromorphic function of $\lambda_1,\lambda_2$ and $\lambda_3$. Then,
the resulting continuous trilinear form
\begin{equation}\label{eqn:trilinear}
\mathbf{T}_{\lambda_1,\lambda_2,\,\lambda_3}:\, V_{\lambda_1}\otimes
V_{\lambda_2}\otimes V_{\lambda_3}\longrightarrow\mathbb C
\end{equation}
is invariant with respect to the diagonal action of $G$:
$$
\mathbf{T}_{\lambda_1,\lambda_2,\,\lambda_3}
\l(\pi_{\lambda_1}(g)\,f_1,\,\pi_{\lambda_2}(g)\,f_2,\,\pi_{\lambda_3}(g)\,f_3\r)
=\mathbf{T}_{\lambda_1,\lambda_2,\,\lambda_3}(f_1,\,f_2,\,f_3).
$$
\end{proposition}
\begin{proof}
We will give in Section \ref{sec:Tconv} a sufficient condition on
$(\lambda_1,\lambda_2,\lambda_3)$ for which the integral
\eqref{eqn:T3l} converges absolutely.
The meromorphic continuation can be justified by the
Atiyah--Bern\-stein--Gelfand regularization of the integral
(\ref{eqn:T3l}) (\cite{xBG}, see also \cite{G}). Parameters $\alpha,\beta$ and
$\gamma$ are chosen in such a way that the integrand in
(\ref{eqn:T3l}) is a section of the volume bundle of $(G/P)^3$.
Whence the invariance follows.
\end{proof}
The case when $P$ is a minimal parabolic subgroup was considered in
\cite{D} for $G = SO_0(m+1,1)$.
We note that in this situation $\overline P$ is
automatically conjugate to $P$.\vskip10pt

Returning to our settings, we have an isomorphism of Lie algebras:
$$
\mathfrak{sp}(1,\mathbb{R})\simeq\mathfrak{so}(2,1)\simeq\mathfrak{sl}(2,\mathbb{R}),
$$
each of which is the `bottom' of different series of Lie algebras,
namely $\mathfrak{sp}(n,\mathbb{R}),\,\mathfrak{so}(n,1)$, and
$\mathfrak{sl}(n,\mathbb{R})$. Bearing this in mind, we list the
following three cases:

Case \textbf{Sp}.  Theorem \ref{thm:gBZ} corresponds to the
evaluation of the trilinear form (\ref{eqn:trilinear}) on the
$K$-fixed vector
$\indic_{\lambda_1}\otimes\indic_{\lambda_2}\otimes\indic_{\lambda_3}$
for the following particular pair: $G = Sp(n,\mathbb{R})$ and
$P=MAN$ a maximal parabolic subgroup such that
$M\simeq\mathbb{Z}/2\mathbb Z \times Sp(n-1,\mathbb{R})$ and $N$ is
the Heisenberg group in $2n-1$ variables. Notice that $S^{2n-1}$ is
a double covering of $G/P$. The representation space $V_\mu$ can be
identified with $V_\mu(\mathbb R^{2n})$ introduced in
(\ref{eqn:Vmu}).
Then the kernel of the operator $\mathcal T_\mu$ introduced in
(\ref{eqn:Tmu}) is $K_\mu(X,Y)=|[X,Y]|^{-\mu-n}\in
V_\mu^{-\infty}\otimes V_\mu^{-\infty}$ which gives rise to the
Knapp--Stein intertwining operator.

Case \textbf{SO}. Theorem \ref{thm:SO} corresponds to the case where
$G=SO_o(m+1,1)$ and $P$ is a minimal parabolic subgroup. Through the
identification $G/P\simeq S^m$ the Knapp--Stein intertwining
operator is given by $\mathcal R_\mu$ (see (\ref{eqn:Rmu})), and the
triple integral in Theorem \ref{thm:SO} corresponds to the
evaluation of the trilinear form (\ref{eqn:T3l}) on the $K$-fixed
vector.

\textbf{ Case: GL}. Yet another expression of the sphere $S^{N-1}$
as a homogeneous space is given by $G/P$, where $=GL(N,\mathbb R)$
and $P$ is a maximal parabolic subgroup corresponding to the
partition $N=1+(N-1)$. The operators $\mathcal Q_\mu$ introduced in
(\ref{eqn:Qmu}) and involved in the Theorem \ref{thm:xy} can also be
interpreted as the Knapp--Stein integrals for representations
induced from $P$ and its opposite parabolic $\overline{P}$. Notice
that the condition C1 fails for $N>2$ and Proposition \ref{prop:51}
does not apply.\vskip10pt

What we have found in particular is the eigenvalues of operators
$\mathcal T_\mu,\,\mathcal Q_\mu$ and $\mathcal R_\mu$ in terms of
Gamma functions. The corresponding eigenspaces are irreducible
representation spaces of the maximal compact subgroup $K$. Indeed,
in all three cases the following condition holds:
\begin{description}
\item C2. The space $K/(K\cap M)$ is a multiplicity--free
             space, in other words, $(K, K\cap M)$ is a Gelfand
             pair.
\end{description}
This implies that the representation space $V_\mu$ contains an
algebraic direct sum of pairwise inequivalent irreducible
representations of $K$ as its dense subspace. Therefore the action
of the operators $\mathcal T_\mu$ on each $K$-representation space
is automatically a scalar multiple of the identity by Schur's lemma.
For example in Case \textbf{Sp}, $K\simeq U(n)$, the corresponding
restriction $\pi_\mu \big|_K$ is given by $ \displaystyle
\bigoplus\limits_{\alpha,\beta\in\mathbb{N}}
        \mathcal{H}^{\alpha,\beta} (\mathbb{C}^n),
$
 and the eigenvalues are described in Theorem
\ref{thm:Teigen}.

In Cases \textbf{SO} and \textbf{GL} the condition C2 is also
satisfied. We can see this by a direct computation but also by the
general observation that the unipotent radical $N$ is abelian and
consequently $(K,\, M\cap K)$ is a symmetric pair.\vskip10pt

Another feature of our settings is the following condition:
\begin{description}
             \item C3. The diagonal action of $G$ on $(G/P)^3$ admits an open orbit.
\end{description}
(In fact, there is only one such an open dense orbit except the case
of $SL(2,\mathbb R)$, where there are two open orbits.)

The condition C3 is connected to the upper bound of the number of
linearly independent trilinear forms for generic
$\lambda_1,\,\lambda_2$ and $\lambda_3$. If this number equals one
then such an invariant trilinear form is proportional to the one
constructed in Proposition \ref{prop:51} under the condition
C1.\vskip10pt

 Case \textbf{Sp} ($n\geq2$) is of a particular interest: the group
 $G$ is of arbitrarily high rank,
$N$ is non-abelian, and $(K,\,M\cap K)$ is a non-symmetric pair.
Nevertheless all the conditions C1, C2 and C3 are fulfilled. The
corresponding trilinear form
$\mathbf{T}_{\lambda_1,\lambda_2,\lambda_3}$ has recently arisen in
a different context, namely in pseudo-differential analysis. More
precisely, a new (non-perturbative) composition formula based on
this trilinear form is established for the Weyl operator calculus on
$L^2(\mathbb R^{2n})$ in \cite{KOPU}, where a slightly different
notation is adopted:
$\mathbf{T}_{\lambda_1,\lambda_2,\,\lambda_3}(f_1,f_2,f_3)=
\mathbf{J}_{-\lambda_1,-\lambda_2;\,\lambda_3}^{0,\,0;\,0}(f_1,f_2,f_3).$

\subsection{Convergence of the invariant triple integral}\label{sec:Tconv}

This subsection provides a sufficient condition for the convergence of
the triple integral in Proposition \ref{prop:51}.

We take $\Sigma^+({\mathfrak {g}}; {\mathfrak {a}})$
 to be the set of weights of ${\mathfrak {n}}$
 with respect to ${\mathfrak {a}}$.
The corresponding dominant Weyl chamber ${\mathfrak {a}}_+^{\ast}$ is defined by
$$
{\mathfrak {a}}_+^{\ast} := \{\nu \in {\mathfrak {a}}^{\ast}: \langle \nu, \alpha \rangle \ge 0
                       \quad
                       \text{for any }\,\, \alpha \in \Sigma^+({\mathfrak {g}};{\mathfrak {a}})\}.
$$

According to the direct sum decomposition
$\mathfrak{a}_{\mathbb{C}}^* = \mathfrak{a}^* + \sqrt{-1} \mathfrak{a}^*$,
we write
$\lambda = \operatorname{Re} \lambda + \sqrt{-1} \operatorname{Im} \lambda$
for $\lambda \in \mathfrak{a}_{\mathbb{C}}^*$.
Then we have
\begin{proposition}\label{prop:Tconv}
Suppose we are in the setting of Proposition \ref{prop:51}.
If
\begin{equation}\label{eqn:Tconv}
\operatorname{Re} \alpha, \operatorname{Re} \beta, \operatorname{Re} \gamma
\in -\rho - \mathfrak{a}_+^*,
\end{equation}
then the integral \eqref{eqn:T3l} converges absolutely.
\end{proposition}
Here, $-\rho-\mathfrak{a}_+^*$ is a subset of $\mathfrak{a}^*$ given
by
$-\rho - \mathfrak{a}_+^*
 := \{ -\lambda-\rho : \lambda \in \mathfrak{a}_+^* \}$.

The rest of this subsection is devoted to the proof of Proposition
\ref{prop:Tconv}.
We will show that the integral kernel
\[
K_{\frac{1}{2}(\alpha-\rho)} (y,z)
K_{\frac{1}{2}(\beta-\rho)} (z,x)
K_{\frac{1}{2}(\gamma-\rho)} (y,z)
\]
is bounded on the triple product manifold
$K \times K \times K$ if the assumption
\eqref{eqn:Tconv} is satisfied.
(As we shall see in Section \ref{sec:Appendix} for specific cases, the condition \eqref{eqn:Tconv} is not a necessary condition for the absolute convergence.)

Consider the multiplication map
\begin{equation*}
     N \times M \times A \times \overline N \to G,
     \quad
     (n,m,a,\overline n) \mapsto n m a \overline n.
\end{equation*}
This is a diffeomorphism into an open dense subset
$
     G':= N M A \overline{N}
$
 of $G$ (the open {\it{Bruhat cell}}).
We define the projection $\mu$ by
$$
  \mu:G' \to {\mathfrak {a}},
  \quad
  n m e^X \overline n \mapsto X.
$$
We set $K' := G' \cap K$.
Then we have
\begin{lemma}
\label{lem:Kmu}
$K'$
 is dense in $K$.
Further, if $\nu \in \mathfrak{a}_+^*$ then
 $\inf_{k \in K'} \langle \nu, \mu (k) \rangle > - \infty$.
\end{lemma}
\begin{proof}
It follows from the Iwasawa decomposition $G=KA \overline N$
 that any element of $G'$ is written as
 $g' = k a \overline n $
 ($k \in K$, $a \in A$, $\overline n \in \overline N$).
Since $G'$ contains the subgroup $A \overline N$,
 we get $k \in K'$.
This leads us to the bijection
 $G'/M A \overline N \simeq K'/M$,
 which then is a dense subset of $G/MA \overline N \simeq K/M$.
Hence,
 $K'$ is dense in $K$.

In order to prove the second assertion,
 we may assume that $G$ is a linear group contained in a connected complex Lie group $G_{\mathbb{C}}$
 with Lie algebra ${\mathfrak {g}} \otimes_{\mathbb{R}} {\mathbb{C}}$.
Let $\overline {P_{\mathbb{C}}} = M_{\mathbb{C}} A_{\mathbb{C}}\overline{N_{\mathbb{C}}}$
 be the complexified parabolic subgroup of $\overline{P}$.
We take a $\theta$-stable Cartan subalgebra ${\mathfrak {t}}$
 of ${\mathfrak {m}}$.
Then,
 ${\mathfrak {h}}= {\mathfrak {t}} + {\mathfrak {a}}$
 is a Cartan subalgebra of ${\mathfrak {g}}$.

We fix a positive set
 $\Delta^+({\mathfrak {g}}_{\mathbb{C}},{\mathfrak {h}}_{\mathbb{C}})$
 of the root system
 such that
$
  \widetilde{\alpha}|_{\mathfrak {a}} \in \Sigma^+({\mathfrak {g}}, {\mathfrak {a}}) \cup \{0\}
$
 for any
$
  \widetilde \alpha \in \Delta^+({\mathfrak {g}}_{\mathbb{C}}, {\mathfrak {h}}_{\mathbb{C}})
$.
Suppose $\widetilde{\lambda} \in {\mathfrak {h}}_{\mathbb{C}}^*$
 is a dominant integral weight
 subject to the following condition:
\begin{equation}
\label{eqn:lmdat}
  \widetilde {\lambda}|_{\mathfrak {a}}=\lambda,
  \quad
  \widetilde {\lambda}|_{\mathfrak {t}}\equiv 0,
  \,\text{ and }\,
  \text{$\widetilde \lambda$ lifts to a holomorphic character of } M_{\mathbb{C}} A_{\mathbb{C}}.
\end{equation}

Then we get a holomorphic character,
to be denoted by $\mathbb{C}_\lambda$,
of $\overline{P_{\mathbb{C}}}$ by extending trivially on
$\overline{N_{\mathbb{C}}}$.

Let ${\mathcal{L}}_{\lambda}^{\mathbb C}:= G_{\mathbb{C}}
\times_{\overline{P_{\mathbb{C}}}}
     \mathbb C_{\lambda}$ be the $G_{\mathbb{C}}$-equivariant holomorphic bundle over
     $G_{\mathbb{C}}/\overline{P_{\mathbb{C}}}$
 associated to the holomorphic character $\mathbb C_{\lambda}$ of $\overline{P_{\mathbb{C}}}$.
Then,
 by the Borel--Weil theorem,
 the space $F_\lambda := {\mathcal{O}} (G_{\mathbb{C}}/\overline{P_{\mathbb{C}}}, {\mathcal{L}}_{\lambda}^{\mathbb C})$
 of holomorphic sections for ${\mathcal{L}}_{\lambda}^{\mathbb C}$
 gives an irreducible finite dimensional representation
 of $G_{\mathbb{C}}$
 with highest weight $\widetilde \lambda$.

Let $f_{\lambda} \in F_{\lambda}$
 be the highest weight vector normalized as $f_{\lambda}(e)=1$.
Then,
we have $f_{\lambda}(n m a \overline n) = a^{-\lambda}$
 for $n m a  \overline n \in N_{\mathbb{C}} M_{\mathbb{C}} A_{\mathbb{C}} \overline{N_{\mathbb{C}}}$.
In particular,
 if $g \in G'$,
we get
$$
  f_{\lambda} (g)=e^{- \langle \lambda, \mu(g)\rangle}.
$$
Since $f_{\lambda}$ is a matrix coefficient,
 $f_{\lambda}|_K$ is a bounded function for any $\lambda$
 coming from the above $\widetilde {\lambda}$.
In light that
$
  {\mathbb{R}}_+ \text{-span} \{\lambda \in {\mathfrak {a}}^* : \widetilde \lambda
  \text{ satisfies \eqref{eqn:lmdat}}\}
$
   equals
  ${\mathfrak {a}}_+^{\ast}$,
we have proved Lemma \ref{lem:Kmu}.
\end{proof}

Let us complete the proof of Proposition \ref{prop:Tconv}. We recall
how the Knapp--Stein integral operator \cite{KS} is given in the
present context. Since we have assumed that the parabolic subgroup
$P$ is conjugate to $\overline{P}$, we can find $w \in K$ such that
$w^{-1} N w = \overline N$. Then,
 we define a function $K_{\lambda}$ defined on an open dense subset of $G \times G$
 by
$$
   K_{\lambda} (g_1, g_2) = e^{\langle \lambda + \rho, \mu(g_1^{-1} g_2 w)\rangle}
   \quad
   \text{ if } g_1^{-1} g_2 w \in G'.
$$
It follows from Lemma \ref{lem:Kmu} below that
 $K_{\lambda}$ is bounded on $K' \times K'$ if $-(\nu + \rho) \in {\mathfrak {a}}_+^{\ast}$,
 and in particular, defines a locally integrable function on $G \times G$.
The distribution kernel of the Knapp--Stein
intertwining operator coincides with $K_{\lambda}(g_1,g_2)$ when $\lambda$ stays in this range.

Return to the setting of Proposition \ref{prop:Tconv},
and assume the condition \eqref{eqn:Tconv}.
Then, by Lemma \ref{lem:Kmu},
we see that $K_{\frac{1}{2}(\alpha-\rho)}(y,z)$ is bounded on
$K \times K$, and likewise for
$K_{\frac{1}{2}(\beta-\rho)}(z,x)$
and
$K_{\frac{1}{2}(\gamma-\rho)}(x,y)$.
Since the integral \eqref{eqn:T3l} is performed over
the product of three copies of the
compact manifold $K/M$ $(\simeq G/P)$,
the integral \eqref{eqn:T3l} converges
absolutely for any $f_j \in V_{\lambda_j}$ $(j=1,2,3)$.
Hence, Proposition \ref{prop:Tconv} has been proved.
\qed

\section{Convergence of the triple integrals}\label{sec:Appendix}

In Section \ref{sec:Tconv}, we have given a sufficient condition for
the absolute convergence of the invariant triple integral in the
general setting. In this section, for the convenience of the reader,
we give the \textit{precise} region of the parameters for which the
triple integrals in our main results converge absolutely. Section
\ref{subsec:6.1} provides a basic machinery for the convergence of
the integral of the product of complex powers under a certain
regularity assumption \eqref{eqn:rIp}. This criterion gives
immediately the precise region of the absolute convergence of the
integral in Theorem \ref{thm:xy} (see Proposition \ref{ex:inner}).
Unfortunately, the regularity assumption \eqref{eqn:rIp} is
fulfilled only for generic points for the triple integral in Theorem
\ref{thm:gBZ}. This difficulty is overcome by additional local
arguments in Section \ref{subsec:6.2} (see Proposition
\ref{ex:symp}).

\subsection{Convergence under the regularity
condition}\label{subsec:6.1}

Let $M$ be a differentiable manifold,
and $f_1,\dots,f_r \in C^\infty(M)$.
We shall always assume that the zero set
$\{ p \in M : f_j(p) = 0 \}$
is non-empty for any $j$ $(1 \le j \le r)$.
For each point $p \in M$,
we define a subset of the index set $\{ 1,\dots,r \}$ by
\[
I(p) := \{ j : f_j(p) = 0 \},
\]
and a non-negative integer by
\[
r(p) := \dim \mathbb{R}\text{-span} \{ df_j(p) : j \in I(p) \}.
\]
Clearly, we have
\[
r(p) \le \# I(p) \ (\le r).
\]
We fix a Radon measure on $M$ which is equivalent to the Lebesgue
measure on coordinating neighbourhoods (i.e.\ having the same sets
of measure zero). Here is a basic lemma for the convergence of the
integral of $|f_1|^{\lambda_1} \cdots |f_r|^{\lambda_r}$ on $M$.
\begin{lemma}\label{lem:converg}
Assume $f_1,\dots,f_r \in C^\infty(M)$ satisfy the following
regularity condition:
\begin{equation}\label{eqn:rIp}
r(p) = \# I(p)
\quad\text{for any $p \in M$}.
\end{equation}
Let $\lambda_1,\dots,\lambda_r \in \mathbb{C}$.
Then, $|f_1|^{\lambda_1} \cdots |f_r|^{\lambda_r}$ is locally
integrable if and only if
\begin{equation}\label{eqn:lmd}
\operatorname{Re} \lambda_j > -1
\quad\text{for any $j$ $(1 \le j \le r)$}.
\end{equation}
\end{lemma}

\begin{remark}
The local integrability does not depend on the choice of our measure
on $M$.
\end{remark}

Here is a prototype of Lemma \ref{lem:converg}:
\begin{example}\label{ex:xj}
Let $M = \mathbb{R}^n$.
We fix $r \le n$, and set $f_j(x) = x_j$ $(1 \le j \le r)$.
Then $|x_1|^{\lambda_1} \cdots |x_r|^{\lambda_r}$ is locally
integrable against the Lebesgue measure $dx_1 \cdots dx_n$ if and only
if\/ $\operatorname{Re} \lambda_j > -1$ $(1 \le j \le r)$.

This assertion is obvious for $r=1$, and the proof for general $r$
is reduced to the $r=1$ case. We observe that the regularity
assumption \eqref{eqn:rIp} is satisfied because $df_j = dx_j$ $(1
\le j \le r)$ are linearly independent at any point $p \in
\mathbb{R}^n$.
\end{example}

The proof of Lemma \ref{lem:converg} is reduced to Example \ref{ex:xj}
as follows:
\begin{proof}[Proof of Lemma \ref{lem:converg}]
Fix a point $p \in M$,
and suppose $I(p) = \{j_1,\dots,j_k \}$.
Then, by the implicit function theorem,
we can find differentiable functions
$y_{k+1},\dots,y_n$ in a neighbourhood $V_p$ of $p$ such that
$\{ f_{j_1},\dots,f_{j_k}, y_{k+1},\dots,y_n \}$
forms coordinates of $V_p$.

Assume \eqref{eqn:lmd} is satisfied.
Then it follows from Example
\ref{ex:xj} that
 the function
$\prod_{j\in I(p)} |f_j|^{\lambda_j}$
is integrable near $p$.
Multiplying it by the continuous function
$\prod_{j \notin I(p)} |f_j|^{\lambda_j}$,
we see that $|f_1|^{\lambda_1} \cdots |f_r|^{\lambda_r}$ is also
integrable near $p$.

Conversely,
assume $|f_1|^{\lambda_1} \cdots |f_r|^{\lambda_r}$ is locally
integrable.
We will show \eqref{eqn:lmd}.
Take a point $p \in M$ such that $f_1(p) = 0$.
By using the above mentioned coordinates in $V_p$,
we can find $q \in V_p$ such that $f_1(q) = 0$ and $f_j(q) \ne 0$
$(2 \le j \le r)$.
Then the integrability of $|f_1|^{\lambda_1} \cdots |f_r|^{\lambda_r}$
near $q$ is equivalent to that of $|f_1|^{\lambda_1}$.
This implies $\operatorname{Re} \lambda_1 > -1$.
Similarly, we get
$\operatorname{Re} \lambda_j > -1$ for all $j$ $(1 \le j \le r)$.
Hence, Lemma \ref{lem:converg} has been proved.
\end{proof}

Next, we discuss the local integrability of the function
$|f_1|^{\lambda_1} \cdots |f_r|^{\lambda_r}$ when the regularity condition
\eqref{eqn:rIp} fails.
The following two examples will be used to determine the range of
parameters for which the triple product integral in Theorem
\ref{thm:gBZ} is absolutely convergent.

\begin{example}\label{ex:xy}
The function $h_\lambda(x,y) := |x|^{\lambda_1} |y|^{\lambda_2}
|x-y|^{\lambda_3}$ is locally integrable on $\mathbb{R}^2$ if and
only if\/ $\operatorname{Re} \lambda_j > -1$ $(1 \le j \le 3)$ and
$\operatorname{Re} (\lambda_1 + \lambda_2 + \lambda_3) > -2$.
\end{example}

\begin{example}\label{ex:xyQz}
Let $p,q > 0$ and $p+q > 2$.
Suppose $Q(z)$ is a quadratic form on $\mathbb{R}^{p+q}$ of signature
$(p,q)$.
Then
$|x|^{\lambda_1} |y|^{\lambda_2} |x+y+Q(z)|^{\lambda_3}$
is locally integrable on $\mathbb{R}^{p+q+2}$ if and only if\/
$\operatorname{Re} \lambda_j > -1$ $(1 \le j \le 3)$.
\end{example}

We observe that  the regularity condition
$r(p) = \# I(p)$ of Lemma \ref{lem:converg} fails at the origin in
both of these examples.
This failure affects the condition on $\lambda_j$
 for the absolute convergence of the integral
 in Example \ref{ex:xy},
  but does not affect in Example \ref{ex:xyQz}.

\begin{proof}[Prof of Example \ref{ex:xy}]
Applying Lemma \ref{lem:converg} to $\mathbb{R}^2 \setminus \{0\}$,
we see that $h_\lambda(x,y)$ is locally integrable on
$\mathbb{R}^2 \setminus \{0\}$ if and only if
$\operatorname{Re} \lambda_j > -1$ $(1 \le j \le 3)$.

To examine the integrability near the origin,
we use the polar coordinate $(x,y) = (r \cos \theta, r \sin \theta)$.
Then we have
\begin{equation}\label{eqn:h2}
h_\lambda(x,y) dx dy
= r^{\lambda_1 + \lambda_2 + \lambda_3 +1}
  h_\lambda (\cos\theta, \sin\theta) dr d\theta.
\end{equation}
Since $\cos\theta$, $\sin\theta$, $\cos\theta-\sin\theta$ do not
vanish simultaneously
 and have simple zero,
\eqref{eqn:h2} is integrable near the origin if and only if
$\operatorname{Re} \lambda_j > -1$ $(1 \le j \le 3)$ and
$\operatorname{Re} (\lambda_1 + \lambda_2 + \lambda_3) > -2$.
Therefore, Example \ref{ex:xy} is proved.
\end{proof}

In order to give a proof of Example \ref{ex:xyQz},
we prepare the following:
\begin{claim}\label{lem:Qdz}
Let $Q(z)$ be as in Example \ref{ex:xyQz}.
Then there exists a continuous function $A(t,\delta)$ of two variables
$t \in \mathbb{R}$ and $\delta \ge 0$ such that
\[
A(t,\delta) \sim c \, \delta^{p+q-2}
\quad\text{as $t \to 0$}
\]
for some positive constant $c$ and that
\[
\int_{|z| \le \delta} g(Q(z)) dz_1 \cdots dz_{p+q}
= \int_{-\delta^2}^{\delta^2} g(t) A(t,\delta) dt.
\]
\end{claim}

\begin{proof}[Proof of Claim \ref{lem:Qdz}]
Without loss of generality,
we may and do assume that $Q(z)$ is of the standard form
$Q(z) = z_1^2 + \dots + z_p^2 - z_{p+1}^2 - \dots - z_{p+q}^2$.
Taking the double polar coordinates
\[
z = (r\omega, s\eta), \ r,s, \ge 0, \  \omega \in S^{p-1}, \  \eta \in S^{q-1},
\]
we have
\begin{equation}\label{eqn:Bdrds}
\int_{|z|<\delta} g(Q(z)) dz
= \operatorname{vol}(S^{p-1}) \operatorname{vol}(S^{q-1})
  \int_{B_+(\delta)} g(r^2-s^2) r^{p-1} s^{q-1} dr ds,
\end{equation}
where we set $B_+(\delta) := \{ (r,s) \in \mathbb{R}^2:
r \ge 0, s \ge 0, r^2+s^2 \le \delta^2 \}$.
By the change of variables
$R := r^2 + s^2$, $t := r^2 - s^2$,
the right-hand side of \eqref{eqn:Bdrds} amounts to
\[
\int_{-\delta^2}^{\delta^2} g(t) A(t,\delta) dt,
\]
where $A(t, \delta)$ is defined by
\[
A(t,\delta)
:= \frac{\operatorname{vol}(S^{p-1}) \operatorname{vol}(S^{q-1})}
        {2^{\frac{p+q}{2}+1}}
   \int_{|t|}^{\delta^2} (R+t)^{\frac{p-2}{2}} (R-t)^{\frac{q-2}{2}} dR.
\]
Putting $t=0$, we have
\[
A(0,\delta)
= \frac{\delta^{p+q-2} \operatorname{vol}(S^{p-1}) \operatorname{vol}(S^{q-1})}
       {2^{\frac{p+q}{2}+1} (p+q-2)}.
\]
Thus, Claim \ref{lem:Qdz} is shown.
\end{proof}

We are ready to complete the proof of Example \ref{ex:xyQz}.
\begin{proof}[Proof of Example \ref{ex:xyQz}]
The regularity condition \eqref{eqn:rIp} is fulfilled except for the
origin.
Applying Lemma \ref{lem:converg} to $\mathbb{R}^{p+q+2} \setminus \{0\}$,
we see that
$|x|^{\lambda_1} |y|^{\lambda_2} |x+y+Q(z)|^{\lambda_3}$ is locally
integrable on $\mathbb{R}^{p+q+2} \setminus \{0\}$ if and only if
$\operatorname{Re} \lambda_j > -1$ $(1 \le j \le 3)$.

What remains to prove is that this function is still integrable near the origin under
the same assumption.
By Claim \ref{lem:Qdz}
we can reduce the convergence of the integral
to that of the three variables case,
namely, it is sufficient to show that
$|x|^{\lambda_1} |y|^{\lambda_2} |x+y+t|^{\lambda_3} A(t,\delta)$
is integrable against $A(t,\delta) dx dy dt$ near $(0,0,0) \in \mathbb{R}^3$
for a fixed $\delta > 0$.
Since $\{ x,y, x+y+t \}$ meets the regularity condition \eqref{eqn:rIp},
we can apply Lemma \ref{lem:converg} again,
and conclude that it is integrable if
$\operatorname{Re} \lambda_j > -1$ $(1 \le j \le 3)$.
Therefore, the proof of Example \ref{ex:xyQz} has been completed.
\end{proof}

\subsection{Applications to the triple integrals}\label{subsec:6.2}

We apply Lemma \ref{lem:converg} to find a condition for
the convergence of the triple
integrals in the previous sections.
The first case is a direct consequence of Lemma \ref{lem:converg}:
\begin{proposition}[see Theorem \ref{thm:xy}]
\label{ex:inner}
Let $M := S^{N-1} \times S^{N-1} \times S^{N-1}$,
and
\[
h_\lambda (x,y,z)
:= | \langle y,z \rangle |^{\lambda_1}
   | \langle z,x \rangle |^{\lambda_2}
   | \langle x,y \rangle |^{\lambda_3}.
\]
Then $\displaystyle \int_M h_\lambda(x,y,z) d\sigma (x) d\sigma (y)
d\sigma (z)$ converges if and only if\/ $\operatorname{Re} \lambda_j
> -1$ $(1 \le j \le 3)$.
\end{proposition}

\begin{proof}[Proof of Proposition \ref{ex:inner}]
Since $h_\lambda (x,y,z)$ is a homogeneous function of $x$ (also,
that of $y$ and $z$), $h_\lambda$ is integrable on $M$ if and only
if it is locally integrable on $\widetilde{M} := (\mathbb{R}^N
\setminus \{0\}) \times
 (\mathbb{R}^N \setminus \{0\}) \times
 (\mathbb{R}^N \setminus \{0\})$.

We set the following functions of $(x,y,z) \in \widetilde{M}$ as
$f_1 := \langle y,z \rangle$,
$f_2 := \langle z,x \rangle$,
$f_3 := \langle x,y \rangle$.
Then we have
\begin{alignat*}{4}
&df_1 ={} && &&\langle z,dy \rangle &&{}+ \langle y,dz \rangle,
\\
&df_2 ={} &&\langle z,dx \rangle && &&{}+ \langle x,dz \rangle,
\\
&df_3 ={} &&\langle y,dx \rangle +{} &&\langle x,dy \rangle.
\end{alignat*}
Let us verify that the functions $f_1$, $f_2$, and $f_3$ meet
the regularity assumption \eqref{eqn:rIp} of Lemma
\ref{lem:converg} on $\widetilde{M}$.

Suppose $\# I(x,y,z) = 3$, namely,
\begin{equation}\label{eqn:xyz0}
\langle y,z \rangle = \langle z,x \rangle = \langle x,y \rangle = 0.
\end{equation}
We will show $r(x,y,z) = 3$.
If not, there would exist
$(a,b,c) \ne (0,0,0)$ such that
$adf_1 + bdf_2 + cdf_3 = 0$,
namely,
\begin{equation}
\label{eqn:abc}
bz + cy = 0,
\
az + cx = 0,
\
ay + bx = 0.
\end{equation}
This would contradict to \eqref{eqn:xyz0}.
Hence $r(x,y,z)$ must be equal to $3$.
Thus \eqref{eqn:rIp} holds when $\# I(x,y,z) = 3$.
Similarly, \eqref{eqn:rIp} can be verified when
$\# I(x,y,z) = 1$ or $2$.
Therefore, the assertion of Proposition \ref{ex:inner} follows from Lemma
\ref{lem:converg}.
\end{proof}

In contrast with Proposition \ref{ex:inner} for the triple product
of complex powers of inner products, the regularity assumption
\eqref{eqn:rIp} does not hold in the symplectic case. The next
example discusses this situation.
\begin{proposition}[see Theorem \ref{thm:gBZ}]
\label{ex:symp}
Let
$M = S^{2n-1} \times S^{2n-1} \times S^{2n-1}$,
and
$$
h_\lambda(X,Y,Z) :=
 |[Y,Z]|^{\lambda_1}  |[Z,X]|^{\lambda_2}  |[X,Y]|^{\lambda_3}.
$$
Then the triple integral $\displaystyle\int_M h_\lambda(X,Y,Z)
d\sigma(X) d\sigma(Y) d\sigma(Z)$ converges absolutely if and only
if
\begin{alignat*}{2}
&\operatorname{Re} \lambda_j > -1
 \ (1 \le j \le 3),
 \ \operatorname{Re} (\lambda_1 + \lambda_2 + \lambda_3) > -2
&&\quad \text{for $n=1$},
\\
&\operatorname{Re} \lambda_j > -1
 \ (1 \le j \le 3),
&&\quad \text{for $n\ge2$}.
\end{alignat*}
\end{proposition}

\begin{proof}[Proof of Proposition \ref{ex:symp}]
Since $h_\lambda(X,Y,Z)$ is a homogeneous function of $X$ (and also,
that of $Y$ and $Z$), $h_\lambda$ is integrable on $M$ if and only
if it is locally integrable on $\widetilde{M} :=
 (\mathbb{R}^{2n} \setminus \{0\}) \times
 (\mathbb{R}^{2n} \setminus \{0\}) \times
 (\mathbb{R}^{2n} \setminus \{0\})$.

Similarly to the way of establishing \eqref{eqn:abc}
 in the proof of Example \ref{ex:inner},
we see that the regularity assumption \eqref{eqn:rIp} for $(X,Y,Z)
\in \widetilde{M}$ fails if and only if the three vectors $X,Y,Z \in
\mathbb{R}^{2n} \setminus \{0\}$ are proportional to each other.
Hence, $h_\lambda$ is locally integrable on $\widetilde{M} \setminus
\{ (a\omega,b\omega,c\omega) : \omega \in S^{2n-1}, \ a,b,c \in
\mathbb{R} \setminus \{0\} \}$ if and only if $\operatorname{Re}
\lambda_j > -1$ $(1 \le j \le 3)$.

Let us find the condition of integrability of $h_\lambda$ near the
point $(X,Y,Z) = (a_0\omega_0, b_0\omega_0, c_0\omega_0)$ for some
$a_0 > 0$, $b_0 \ne 0$, $c_0 \ne 0$, and $\omega_0 \in S^{N-1}$.
For this we take coordinates as
\begin{align*}
&X = a\omega,
\\
&Y = b\omega + xJ\omega + u,
\\
&Z = c\omega + yJ\omega + v,
\end{align*}
where $a,b,c,x,y \in \mathbb{R}$,
$\omega \in S^{2n-1}$,
and
$u,v \in (\mathbb{R}\text{-span}\{ \omega, J\omega \})^\perp$
$(\simeq \mathbb{R}^{2n-2})$,
and we consider the case where
$a-a_0, b-b_0, c-c_0$,
$x,y,u,v$ are near the origin.
In view of the relations $[\omega,J\omega] = -1$ and
$
\langle \omega,u \rangle = \langle \omega,v \rangle =
\langle J\omega,u \rangle = \langle J\omega,v \rangle = 0$,
we have
\begin{align}\label{eqn:bycx}
&h_\lambda(X,Y,Z) dX dY dZ
\nonumber
\\
&= | -by + cx + [u,v] |^{\lambda_1} |ay|^{\lambda_2} |ax|^{\lambda_3}
   a^{2n-1} da db dc d\sigma(\omega) dx dy du dv.
\end{align}
Since we are dealing with the local integrability for
$a,b,c \ne 0$,
the main issue is the local integrability against $dxdydudv$.

First, suppose $n=1$.
Then \eqref{eqn:bycx} is locally integrable if and only if
$|-by + cx|^{\lambda_1} |y|^{\lambda_2} |x|^{\lambda_3} dx dy$
is locally integrable on $\mathbb{R}^2$ for fixed
$b,c \ne 0$.
By Example \ref{ex:xy},
this is the case if and only if
$\operatorname{Re} (\lambda_1 + \lambda_2 + \lambda_3) > -2$
in addition to $\operatorname{Re} \lambda_j > -1$ $(1 \le j \le 3)$.

Second, suppose $n \ge 2$.
Then $[u,v]$ is a quadratic form on $\mathbb{R}^{4n-4}$ of signature
$(2n-2, 2n-2)$.
By Example \ref{ex:xyQz},
$|-by + cx + [u,v]|^{\lambda_1} |y|^{\lambda_2} |x|^{\lambda_3} dx dy
du dv$
is locally integrable on
$\mathbb{R} \times \mathbb{R} \times \mathbb{R}^{2n-2} \times
 \mathbb{R}^{2n-2}$
if and only if $\operatorname{Re} \lambda_j > -1$ $(1 \le j \le 3)$
and this estimate is locally uniform with respect to
$a,b,c$ $(\ne 0)$ and $\omega \in S^{2n-1}$.
Hence, the right-hand side of \eqref{eqn:bycx} is locally integrable
if $\operatorname{Re} \lambda_j > -1$ $(1 \le j \le 3)$.

Thus, the proof of Proposition \ref{ex:symp} is completed.
\end{proof}

\begin{proposition}
\label{prop:6.9}
Let $M = S^m\times S^m\times S^m$ and
\[h_\lambda (X,Y,Z) = \vert Y-Z\vert^{\lambda_1} \vert Z-X\vert^{\lambda_2} \vert Z-Y\vert^{\lambda_3}\ .
\]
Then the triple integral $\int_M h_\lambda(X,Y,Z) d\sigma(X)d\sigma(Y)d\sigma(Z)$ converges if and only if
\[ {\rm Re}\, \lambda_j > -m \quad (1 \le j \le 3),\quad {\rm Re} \, (\lambda_1+\lambda_2+\lambda_3) >-2m\ .
\]
\end{proposition}

The proof follows the same lines as before.

\bigskip
\noindent \textbf{Acknowledgement.}
     The second author is partially supported by Grant-in-Aid
     for Scientific Research (B) (18340037), Japan
     Society for the Promotion of Science, and the
     Alexander von Humboldt Foundation.

\footnotesize{

}
\bigskip
\footnotesize{ \noindent Addresses: (JLC) Institut \'Elie Cartan
(CNRS UMR 7502), Universit\'e Henri Poincar\'e Nancy 1,
B.P. 70239, F-54506 Vandoeuvre-l\`es-Nancy, France.\\
(TK) Graduate School of Mathematical Sciences, The University of
 Tokyo,
3-8-1 Komaba, Meguro, Tokyo, 153-8914 Japan.\\
(B\O ) Matematisk Institut, Byg.\,430, Ny Munkegade, 8000 Aarhus C,
Denmark.\\
(MP) Laboratoire de Math\'ematiques, (CNRS FRE 3111), Universit\'e
de Reims, B.P. 1039,  F-51687 Reims, France.\bigskip

\noindent \texttt{{Jean-Louis.Clerc@iecn.u-nancy.fr,
 toshi@ms.u-tokyo.ac.jp,
 orsted@imf.au.dk,
 pevzner@univ-reims.fr.}}

}

\begin{thebibliography}{99}
\bibitem{B}
W. N. Bailey, \emph{Generalized Hypergeometric Series}, Cambridge
University Press, 1935.

\bibitem{Bc}
W. Beckner, Sharp Sobolev inequalities on the sphere and the
Moser--Trudinger inequality, \emph{Ann. of Math.}, \textbf{138},
(1993), pp. 213--242.

\bibitem{xBG}
I. N. Bernstein and S. I. Gelfand,
Meromorphic property of the functions $P^\lambda$,
Funktsional. Anal. i Prilozhen.
\textbf{3} (1969), 84--85.

\bibitem{B1}
J. Bernstein and A. Reznikov, Analytic continuation of
representations and estimates of automorphic forms, \emph{Ann. of
Math.} (2) \textbf{150} (1999), pp. 329--352.

\bibitem{B2}
J. Bernstein and A. Reznikov, Estimates of automorphic forms,
\emph{Mosc. Math. J.} \textbf{4} (2004), pp. 19--37.

\bibitem{D}
A. Deitmar, Invariant triple products, \emph{Int. J. Math. Sci}.
2006, 22 pp.
\href{http://doi:10.1155/IJMMS/2006/48274}{%
doi:10.1155/IJMMS/2006/48274}.


\bibitem{G}
I. M. Gelfand, G. Shilov, \emph{Generalized Functions}. vol. 1,
Academic Press, 1964.

\bibitem{GR}
I.~S. Gradshteyn and I.~M. Ryzhik, \emph{Table of integrals, series, and
  products}, Academic Press, New York, 1965.

\bibitem{xhel}
S. Helgason, \emph{Groups and Geometric Analysis,
Integral Geometry, Invariant Differential Operators, and Spherical Functions},
Academic Press, Inc. 1984.

 \bibitem{KS}
 A.~W. Knapp, E.~M. Stein, Intertwining operators for semisimple groups. II.
 \emph{Invent. Math}. \textbf{60} (1980), no. 1, pp. 9--84.

\bibitem{xAnn}
T. Kobayashi,
Discrete decomposability of the restriction of $A_{\mathfrak{q}}(\lambda)$ with
respect to reductive subgroups II:
Micro-local analysis and asymptotic $K$-support,
\href{http://dx.doi.org/10.2307/120963}{%
\emph{Ann. Math}}. \textbf{147} (1998), 709--729.

\bibitem{KM-intopq} T. Kobayashi and G. Mano,
{Integral formula of the unitary inversion operator for the minimal representation of $O(p,q)$},
\href{http://projecteuclid.org/euclid.pja/1176126886}
{\textit{Proc. Japan Acad. Ser. A}} {\bf 83} (2007), 27--31;  the full paper (to appear in the Mem. Amer. Math. Soc.) is available at
\href{http://uk.arxiv.org/abs/0907.3749}{arXiv:0712.1769}. 

\bibitem{KOPU}
  T. Kobayashi, B. \O rsted, M. Pevzner, and A. Unterberger,
  Composition formulas in the Weyl calculus, \href{http://dx.doi.org/10.1016/j.jfa.2008.12.023}{\emph{J. Funct. Anal.}} \textbf{257},
  (2009), pp. 948--991.

\bibitem{P}
R. A. Proctor, Odd symplectic groups, \emph{Invent. Math.}
\textbf{92}, (1988), pp. 307--332.


\bibitem{xsabbah}
C. Sabbah,
Polyn\^omes de Bernstein--Sato \`a plusieurs variables,
S\' eminaire \' Equations aux d\' eriv\' ees partielles (Polytechnique), 1986--1987, exp. 19.


\bibitem{VK}
N. Ya. Vilenkin, A.U. Klimyk,\emph{ Representations of Lie groups
and special functions. Vol. 2. Class I Representations, special
functions and integral transforms.} Mathematics and its Applications
(Soviet Series), 74. Kluwer Academic Publishers Group, Dordrecht,
1993.


\end{thebibliography}
\end{document}